\documentclass[12pt,a4paper,oneside]{article}
\usepackage{amsmath,amssymb,amsthm,color,amscd,amsfonts}
\usepackage{amsopn}
\usepackage[a4paper,margin=1in]{geometry}
\usepackage{dsfont}

\usepackage[T1]{fontenc}

\usepackage[bookmarks=true,unicode,colorlinks,urlcolor=red,citecolor=red,linkcolor=blue]{hyperref}
\usepackage{ushort}

\reversemarginpar

\def\dist{\mathop{\rm dist}\nolimits}
\def\diam{\mathop{\rm diam}\nolimits}

\newtheorem{thm}{Theorem}[section] 
\newtheorem{lem}[thm]{Lemma} \newtheorem{prop}[thm]{Proposition}
\newtheorem{cor}[thm]{Corollary}
\newtheorem*{rem*}{Remark} 

\numberwithin{equation}{section}
\newcommand{\scalp}[2]{#1\cdot#2}
\newcommand{\gener}{{\cal L}}
\newcommand{\domgen}{{\cal D}}

\newcommand{\KatoC}{{\cal J}}
\newcommand{\Fourier}[1]{\mathcal F (#1)}
 \newcommand{\R}{\mathds{R}}
 
\newcommand{\Rd}{{\R^{d}}}

\newcommand{\N}{\mathds{N}}

\renewcommand{\leq}{\leqslant} 
\renewcommand{\geq}{\geqslant}

\newcommand{\indyk}[1]{\mathds{1}_{#1}}
\newcommand{\ind}[1]{{\bf 1}_{#1}}

\usepackage{marginnote}

\title{Heat kernels of non-local Schr\"odinger operators with Kato potentials}

\author{Tomasz Grzywny, Kamil Kaleta, and Pawe{\l} Sztonyk\\
  Wroc\l{}aw University of Science and Technology \\ Wybrzeże Wyspiańskiego 27, 50-370 Wrocław, Poland}

\begin{document}
\maketitle
\footnotetext{ Sztonyk were supported by the National Science Centre, Poland, grant no. 2017/27/B/ST1/01339. K. Kaleta was supported by the National Science Centre, Poland, grant no. 2019/35/B/ST1/02421.\\

\emph{e-mail address}: tomasz.grzywny@pwr.edu.pl, kamil.kaleta@pwr.edu.pl, pawel.sztonyk@pwr.edu.pl  }

 
\begin{abstract}
 We study heat kernels of Schr\"odinger operators whose kinetic terms are non-local operators built for sufficiently regular symmetric L\'evy measures with radial decreasing profiles and potentials belong to Kato class. Our setting is fairly general and novel -- it allows us to treat both heavy- and light-tailed L\'evy measures in a joint framework.  
We establish a certain relative-Kato bound for the corresponding semigroups and potentials. This enables us to apply a general perturbation technique to construct the heat kernels and give sharp estimates of them. Assuming that the L\'evy measure and the potential satisfy a little stronger conditions, we additionally obtain the regularity of the heat kernels.  Finally, we discuss the applications to the smoothing properties of the corresponding semigroups. Our results cover many important examples of non-local operators, including \emph{fractional} and \emph{quasi-relativistic} Schr\"odinger operators. 
\end{abstract}

\section{Introduction}
Let $d\in\N.$ We consider non-local operators
$$
  \gener\varphi(x) = \int_{\Rd} \left(\varphi(x+y)-\varphi(x)-\nabla \varphi(x)\cdot y \ind{B(0,1)}(y)\right)\, \nu(dy), \quad \varphi\in C^2_c(\Rd),
$$
where $\nu$ is a given L\'evy measure, i.e.\ $\int (1\wedge |y|^2)\,\nu(dy)<\infty.$  We assume here that
$\nu$ is symmetric (i.e.\ $\nu(-A)=\nu(A)$, for every Borel set $A\subset\Rd$) and absolutely continuous with a density $\nu(x)$   which has a sufficiently regular profile.  More precisely, we require that there exist a continuous, positive and decreasing function $f:\: (0,\infty)\to (0,\infty)$ and the positive constants $C_1,C_2$ such that $$
C_1 f(|x|) \leq \nu(x)\leq C_2 f(|x|), \quad x\in\Rd\setminus\{ 0 \},
$$ 
and $f$ satisfies regularity conditions (A) and (B) stated below. 

The goal of this article is to prove existence, estimates and regularity of heat kernels for non-local Schr\"odinger operators 
$\gener +q$, where the potential $q$ belongs to  the Kato class corresponding to $\gener$. 

We  impose  the following two  assumptions on the profile $f$. 

\bigskip

\begin{itemize}
\item[\textbf{(A)}]  
There exist a constant $C_3>0$ such that
\begin{equation*}
  \int_{\R^d} f_{1}(|x-y|) f_{1}(|y|) dy \leq C_3 \, f(|x|), \quad |x| \geq 2,
\end{equation*}
where $f_{1}(r) = \indyk{(1,\infty)}(r)f(r)$.
\end{itemize}

\bigskip

\begin{itemize}
\item[\textbf{(B)}]  

The function $s\to s^d f(s)$ is decreasing on $(0,2]$
and 
there exist constants $C_4,C_5 > 0$ and
$ 0<\alpha_1\leq \alpha_2<2 \wedge d$
such that    
\begin{equation}\label{eq:doubling_kappa}
	C_4 \left(\frac{R}{r}\right)^{d+\alpha_1} \leq \frac{f(r)}{f(R)} \leq 
	C_5 \left( \frac{R}{r} \right)^{d+\alpha_2},\quad 2 \geq R\geq r >0.
\end{equation}
\end{itemize}
\medskip
\noindent
 These assumptions cover a large class of profiles that lead to many important non-local operators. The key examples are the \emph{fractional Laplacian} $\gener = -(-\Delta)^{\alpha/2}$ and the corresponding \emph{relativistic operators} $\gener = -(-\Delta+m^{2/\alpha})^{\alpha/2}+m$, where $\alpha \in (0,2)$ and $m>0$. These and further examples will be discussed later in more detail.    

We say that a Borel function 
$q:\: \Rd \to \R$ belongs to the Kato class $\KatoC$ corresponding to the L\'evy operator $\gener$ 
if
\begin{equation*}
	  \lim_{\delta\to 0^+} \sup_{x\in\Rd} \int_{|y-x|<\delta} \frac{1}{|y-x|^{2d} f(|y-x|)} |q(y)|\, dy = 0.
	\end{equation*}
The following theorem is  the first main result  of the paper.

\begin{thm}\label{MTh} Let $q\in\KatoC$ and let (A) and (B) hold. 
Then there exists a function 
$\widetilde{p}: (0,\infty)\times \Rd \times \Rd \to (0,\infty)$ such that
 \begin{align} \label{eq:weak_sol}
  \int_0^\infty \int_{\Rd} \widetilde{p}(t,x,y)\left( \partial_t \phi(s+t,y) + \gener\phi(s+t,y) + q(y)\phi(s+t,y) \right)\, dydt = -\phi(s,x)
\end{align}
for every $x\in\Rd,$ $s>0$ and $\phi \in C_c^\infty(\R\times \Rd)$. The kernel $\widetilde{p}(t,x,y)$ 
for every $t>0$ is a jointly continuous  and symmetric function of $(x,y)$, satisfies the
semigroup property
\begin{equation}\label{CH-K}
  \int_{\Rd} \widetilde{p}(s,x,z) \widetilde{p}(t,z,y)\, dz = \widetilde{p}(t+s,x,y),\quad x,y,\in \Rd, t,s>0,
\end{equation}
and for every $\eta\in (0,1)$ 
there exists $h_\eta>0$ such that for every $m\in\N$ and $0<t<mh_\eta$ we have
\begin{equation}\label{tildep_est}
 (1-\eta)^m \leq \frac{\widetilde{p}(t,x,y)}{p(t,y-x)} \leq \frac{1}{1-\eta}\exp{\frac{\eta t}{h_{\eta}(1-\eta)}},\quad x,y\in\Rd,
\end{equation}
where $p$ is a heat kernel of the operator $\gener$.
\end{thm}

We  remark that sharp, finite-time horizon, estimates of the kernel $p$ are known, see Lemma \ref{p_t_est} below. Consequently, from \eqref{tildep_est} we obtain sharp and explicit estimates for $\widetilde{p}$.

Our proof of Theorem \ref{MTh} is based on a general perturbation technique which was developed in \cite{BogdanHansenJakubowski2008} -- the kernel $\widetilde p$ is constructed by means of perturbation series.  That paper gives sufficient conditions
for an initial semigroup and a potential in order to get a new perturbed kernel and obtain its estimates.  The condition we have to check here, called a \emph{relative Kato condition}, is given by \eqref{eq:RelativelyKato}. 
This bound is easy to verify if the  kernel of the original semigroup satisfies globally the 3G inequality (see e.g.\ \cite[Theorem 4]{BogdanHansenJakubowski2008}; cf.\ its local version in Lemma \ref{lem:3G}) which is the case if $f$ decays polynomially at infinity. In our setting, which allows for much faster decay (including exponential one), the 3G inequality generally does not hold and therefore we need to establish this condition in a different manner. It is done in Theorem \ref{th:VeryIT} which is the key technical step in Section \ref{sec:perturbation}. The proof of Theorem \ref{MTh} is concluded in Section \ref{sec:solution} where we show that the kernel $\widetilde p$ obtained in this way is indeed a solution to \eqref{eq:weak_sol} and we prove its continuity.

Semigroups and heat kernels for Schr{\"o}dinger operators involving Laplacian $\Delta$ are now classical topics -- they were studied in many papers by both analytic and probabilistic methods, see e.g.\ \cite{Simon1982,LiskevichSemenov1998,Zhang2003,BogdanDziubanskiSzczypkowski2019}, to mention just a few contributions in this field. We refer the reader to \cite{BogdanDziubanskiSzczypkowski2019} for an excellent overview of the results concerning the heat kernels of the classical Schr\"odinger operators with Kato class and more singular potentials.

Similar problems have also been investigated for Schr\"odinger operators involving the fractional Laplacian $-(-\Delta)^{\alpha/2}$ and potentials from the corresponding Kato class. 
Estimates of heat kernels for such semigroups were 
obtained in \cite{Song2006} by probabilistic methods and in \cite{BogdanHansenJakubowski2008} by perturbation techniques. 
 The results gradually extended in
the papers \cite{Jakubowski2009,BogdanJakubowskiSydor2012,BogdanHansenJakubowski2013,WangZhang2015,BogdanButkoSzczypkowski2016,BogdanGrzywnyJakubowskiPilarczyk2019,JakubowskiWang2020} to more general semigroups, potentials, spaces and time-dependent cases.

 The novelty of Theorem \ref{MTh} is that it covers non-local operators $\gener$ with L\'evy densities fast decaying at infinity, including exponential and stretched exponential ones, e.g.\ relativistic operators $\gener = -(-\Delta+m^{2/\alpha})^{\alpha/2}+m$. To the best of our knowledge, no similar result has been obtained for Schr\"odinger operators based on such $\gener$'s with Kato class potentials.   
For example, for every $\alpha\in(0,2)$ the following functions 
\begin{equation}\label{eq:def_of_f}
	f(r) = \ind{(0,1]}(r)\cdot r^{-\alpha-d} + 
           e^m \ind{(1,\infty)}(r) \cdot e^{-mr^{\beta}}r^{-d-\eta},
\end{equation} 
where 
$$ 
  m=0 \text{ and } \eta >0,
$$
or
$$
  m>0 \text{ and } \beta\in (0,1) \text{ and } d+\eta \geq 0,
$$
or
$$
  m>0 \text{ and } \beta = 1 \text{ and } \frac{d-1}{2}+\eta > 0,
$$
satisfy the condition (A), which was proved in \cite[Proposition 2]{KaletaSztonyk2017} (see also \cite[Section 3.1]{KaletaSchilling2020} for some general easy-to-check sufficient conditions). In particular,
the both conditions (A) and (B) with the profile function \eqref{eq:def_of_f} hold for the  quasi-relativistic operators,  where we have $\alpha_2=\alpha_1=\alpha,$ $\beta=1,$ and $\eta=\frac{\alpha+1-d}{2}$ (see \cite{Ryznar2002,KulczyckiSiudeja2006}).  Of course they are satisfied also
in the well-known case of the fractional Laplacian (i.e. $\nu(x)=c|x|^{-\alpha-d}$) with $\eta=\alpha_2=\alpha_1=\alpha$ and $m=0.$
 Examples for which $\alpha_1<\alpha_2$ can be found in \cite{Schilling}. We remark in passing that Schr\"odinger semigroups corresponding to non-local operators satisfying (A) have been widely studied for confining potentials, see \cite{KaletaSchilling2020} and further references in that paper. 

Before we proceed to present our second main result, we need some more preparation.
Let $g(s)=s^d f(s).$ If (B) holds then the function $g$ is continuous and decreasing on $(0,2]$ so there exists a decreasing inverse function $g^{-1}: [g(2),\infty) \to (0,2].$ Let $H: (0,\infty) \to (0,\infty)$
be a continuous, increasing function such that $H(t)=g^{-1}\left(\frac{1}{t}\right),$ for
$t\in \left(0,\frac{1}{g(2)}\right].$
For $n>0$, we let
$$
  G_n(t,x) = \min\left\{H(t)^{-d},tf\left(\tfrac{|x|}{4}\right)\right\} + H(t)^{-d}\left(1+\tfrac{|x|}{H(t)}\right)^{-n}, \quad t>0, \ x\in\Rd.
$$

We will say that $q\in\KatoC_a$ for some $a\geq 1$ if 
  \begin{equation}\label{extraKato}
	\lim_{\delta\to 0^+} \sup_{x\in\Rd} \int_{|y-x|<\delta} \frac{1}{|y-x|^{d(1+1/a)} f(|y-x|)^{1/a}} |q(y)|\, dy = 0.
\end{equation}
It follows from (B) that $g(s)=s^df(s)\geq c s^{-\alpha_1}>1,$ for $s$ sufficiently small, hence
we have $\KatoC_a \subset \KatoC_b$ for all $a \geq b \geq 1.$ 
This allows also one to give a sufficient $L^p$-condition which is very useful in applications. 
Indeed, by combining this estimate with the H\"older inequality, we can easily show that for every $p >(ad)/\alpha_1$ the inclusion $L^p(\R^d)+L^{\infty}(\R^d) \subseteq \KatoC_a$ holds. Furthermore it follows from (B) that
if $q\in \KatoC_a$ then the dilation $q(\kappa \cdot)$ also belongs to $\KatoC_a$ for every $\kappa>0.$

 Kato classes are usually defined in the literature in terms of heat kernels.   
	We use here the above simpler condition for $\KatoC_a$ and 
	in Corollary \ref{CorKato_Eq} below we prove the equivalence of definitions.  
We remark that very similar Kato-type spaces were investigated in \cite{KuwaeMori2020}  (cf.\ also \cite[Section B.3]{Simon1982} and \cite{Guneysu2022} for the case of Laplacian).

We need also an additional assumption on regularity of the profile $f$.
\begin{itemize}
\item[\textbf{(C)}]  
There exists a constant $C_6>0$ such that 
\begin{equation*}
  \frac{f(r+\kappa)}{f(r)} \leq C_6 \frac{f(s+\kappa)}{f(s)}, \quad s>r>0,\, \kappa>0.
\end{equation*}
\end{itemize}
We note that (C) holds for every $f$ such that $\log f$ is convex. 
It holds also for $f$ such that (B) is satisfied
globally, i.e.\
for all $R>r>0.$ Of course, every function $f$ given by $\eqref{eq:def_of_f}$ also satisfies (C).   

The following theorem is the second main result of the paper.

\begin{thm}\label{THE2} Let (A), (B) and (C) hold. Let $n>0$.
If $q\in\KatoC_a$ for some $a\in (1,2]$,
then for every $t_0 >0$ there exists $C_1=C_1(n)$ such that
\begin{equation}\label{Hoelkernel}
  |\widetilde{p}(t,z,y) - \widetilde{p}(t,x,y)| \leq C_1  
	   \left(\tfrac{|z-x|}{H(t)}\wedge 1\right)^{\tfrac{\alpha_1}{b}}\left(G_n(t,y-z) + G_n(t,y-x)\right),
\end{equation}
for all $t\in (0,t_0), x,y\in\Rd$, where $\frac{1}{b} = 1 - \frac{1}{a}.$

If, in addition, $\alpha_1 >1$ and $q\in \KatoC_a$ for some $a>\frac{\alpha_1}{\alpha_1 - 1}$, then for $t>0$ and $y \in \R^d$ the function $\widetilde{p}(t,\cdot,y)$ has all partial derivatives everywhere and for every $t_0>0$ there exists $C_2$ such that for every $i\in\{1,...,d\}$ we have
	\begin{equation}\label{DerEst}
	  \left|\frac{\partial}{\partial x_i} \widetilde{p}(t,x,y) \right| \leq C_2 H(t)^{-1} G_n(t,y-x), \quad t\in (0,t_0), \ x,y\in\Rd.
	\end{equation}
\end{thm}

\medskip
\noindent
We note that the regularity obtained in Theorem \ref{THE2} seems to be new even in the case of fractional Laplacian.

We now give some applications of our main results. By \eqref{CH-K} and \eqref{tildep_est} of Theorem \ref{MTh} we see that the kernels $\widetilde{p}(t,x,y)$ define
the semigroup of bounded operators on each $L^p(\R^d)$, $1 \leq p \leq \infty$, which are given by
$$
\widetilde{P}_t \varphi(x) = \int_{\R^d}  \widetilde{p}(t,x,y)\varphi(y) dy, \quad t >0.
$$
Let $\beta \in (0,1)$. Denote by $C^{0,\beta}(\R^d)$ the space of $\beta$-H\"older continuous functions on $\R^d$ with its seminorm
$$
\left\|\varphi\right\|_{C^{0,\beta}}:= \sup_{ x,y \in \R^d \atop x \neq y} \frac{|\varphi(x) - \varphi(y)|}{|x-y|^{\beta}}.
$$
Our next result follows directly from Theorem \ref{THE2}.

\begin{cor} \label{cor:reg_sem}
Let (A), (B) and (C) hold. 
Then we have the following statements. 
\begin{itemize}
\item[(1)] Let $\beta \in (0,\alpha_1/2]$. If $q\in\KatoC_a$ for some $a \in [\alpha_1/(\alpha_1-\beta),2]$, then for every $t>0$ and $p \in [1,\infty]$ the operator $\widetilde{P}_t$ maps $L^p(\R^d)$ continuously into $C^{0,\beta}(\R^d)$, and for any $t_0>0$ there exists a constant $C=C(p,t_0)$ such that 
$$
\left\|\widetilde{P}_t\right\|_{L^p \rightarrow C^{0,\beta}} \leq C H(t)^{-d/p-\beta}, \quad t \in (0, t_0).
$$

\item[(2)] If $\alpha_1 >1$, $q\in \KatoC_a$ for some $a>\alpha_1/(\alpha_1 - 1)$ and $\varphi \in L^p(\R^d)$ for some $p \in [1,\infty]$, then for every $t > 0$ the function $\widetilde{P}_t \varphi$ has all partial derivatives everywhere in $\R^d$ and for every $t_0 >0$ there exists a constant $C=C(p,t_0)$ such that
	$$
	   \left\|\nabla \widetilde{P}_t \right\|_{L^p \rightarrow L^{\infty}} \leq C H(t)^{-d/p-1}, \quad  t \in (0,t_0).
	$$ 
\end{itemize}
\end{cor}
We remark in passing that it also follows from \eqref{tildep_est} that all $\widetilde{P}_t$'s are bounded operators from $L^p(\R^d)$ to $L^q(\R^d)$, for every $1 \leq p < q \leq \infty$. In combination with Corollary \ref{cor:reg_sem}(1), it immediately implies that for every $t>0$ and $1 \leq p \leq q \leq \infty$ the operator $\widetilde{P}_t$ maps $L^p(\R^d)$ continuously into the Banach space $C^{0,\beta}(\R^d)\cap L^q(\R^d)$ (with the norm $\left\|\varphi\right\|_{C^{0,\beta} \cap L^q} := \left\|\varphi\right\|_{C^{0,\beta}} + \left\|\varphi\right\|_{L^q}$), whenever $q\in\KatoC_a$ for some $a \geq \alpha_1/(\alpha_1-\beta)$.

 We refer the reader to \cite[Theorem B.3.5]{Simon1982} and \cite[Theorem 11.7]{LiebLoss2001} for results on H\"older smoothing properties for classical Schr\"odinger semigroups on bounded domains in $\R^d$. The global results for magnetic Schr\"odinger semigroups have been obtained just recently in \cite{FurstGuneysu2021}. 

We use c, C, M (with subscripts) to denote finite positive constants which depend only on $f$, $\nu$, and the dimension $d$. Any additional dependence is explicitly indicated by writing, e.g., $c =c(n).$ 
We write $f(x) \approx g(x)$ to indicate that there is a constant $c$ such that 
$c^{-1}f(x)\leq g(x) \leq c f(x).$

\section{Preliminaries}

There exists a corresponding to $\gener$ probabilistic convolution semigroup of measures $\{\mu_t\}_{t\geq 0}$ on $\Rd$, 
such that
$$
  \Fourier{\mu_t}(u) = \int_{\Rd} e^{i\scalp{u}{y}}\,\mu_t(dy) = e^{-t\Phi(u)},\quad t\geq 0, u\in\Rd,
$$
where 
\begin{equation*} 
  \Phi(u) =   \int \left(1-\cos(\scalp{u}{y}) \right)\nu(y)\, dy, \quad u\in\Rd.
\end{equation*}
The strongly continuous operator semigroup $\{P_t,\, t\geq 0\}$ corresponding to $\{\mu_t,\, t\geq 0\}$ is given by
$P_t\varphi (x) = \int \varphi(x+y)\,\mu_t(dy),$ where $\varphi\in C_\infty(\Rd).$  
The generator $\tilde\gener$ of this semigroup is,
as usually, defined as the strong limit 
$$
  \tilde{\gener} \varphi (x) = \lim_{t\to 0^+} \frac{P_t\varphi(x) - \varphi(x)}{t},
$$
for all $\varphi\in \domgen = \{\varphi\in C_\infty(\Rd):\: \lim_{t\to 0^+} \frac{P_t\varphi(x) - \varphi(x)}{t} \text{  exists uniformly in $x\in\Rd$.}\}.$
We have $C^2_c(\Rd)\subset \domgen$ and
$\tilde{\gener}\varphi=\gener\varphi$ for all $\varphi\in C^2_c(\Rd)$ (see, e.g., \cite{BoettcherSchillingWang2013}).

It follows from (B) that there exists a constant $c$ such that
\begin{equation}\label{belowPhi}
  \Phi(u) \geq (1-\cos 1) \int_{|y|\leq \frac{1}{|u|}}  |\scalp{u}{y}|^2 \nu(y)\,dy \geq c |u|^{\alpha_1},\quad |u|\geq 1, 
\end{equation}
hence for every $t>0$ there exists a smooth transition density $p(t,x),$ such that $\mu_t(dx)=p(t,x)\,dx.$ Furtheremore $p(t,x,y)=p(t,y-x)$ is a heat kernel for $\gener$.

Let
$$
  \Psi(r)=\sup_{|\xi|\leq r} \Phi(\xi),\quad r>0.
$$
It follows from \cite[Lemma 6]{Grzywny2014} (or \cite[Proposition 1]{KaletaSztonyk2015})  
that
\begin{equation}\label{Pineq}
  \Psi(r) \approx \int_{\Rd} \left( 1\wedge |ru|^2 \right) \nu(u)\, du,\quad r>0.
\end{equation}
This yields $\Psi(r) \geq c r^2 \int_0^{\tfrac{1}{r}} s^{d+1}f(s)\, ds $, for $r>0,$ and if 
  $\int_{\Rd} |x|^2 \nu(x)\, dx < \infty,$ then 
	$\Psi(r) \approx r^2 $, for $ r\leq 1$.
	
	We note that $\Psi$ is continuous and non-decreasing and $\sup_{r>0} \Psi(r)=\infty$ (see \eqref{belowPhi}).
Let 
$$
  \Psi_-(s)=\sup\{r>0: \Psi(r)=s\} \quad \text{for} \ s\in (0,\infty),
$$ 
so that $\Psi(\Psi_-(s))=s$ for $s\in (0,\infty)$ and $\Psi_-(\Psi(s))\geq s$ for $s>0$. It is 
also easy to see that 
\begin{equation}\label{OnPsi_}
  \text{if   }  r<\Psi(s) \quad  \text{  then   } \quad  \Psi_-(r) \leq s, \quad r,s>0. 
\end{equation}	
	To shorten the notation below, we set 
\begin{equation*}
h(t):= \frac{1}{\Psi_-\left(\frac{1}{t}\right)}.
\end{equation*}
Notice that the both functions $\Psi_-$ and $h$ are increasing on $(0,\infty)$. 

	 In the following four lemmas we do use the assumption (B) without using
	$\alpha_2<d,$ so in fact they hold also for $d=1$ and $\alpha_2>1.$ 

\begin{lem}\label{lem:doubling_h} If $(B)$ holds then there exist constants $ C_1,C_2,$ such that
  \begin{equation}\label{asympt_Psi}
	  C_1 r^{-d} f\left(\tfrac{1}{r}\right) \leq \Psi(r) \leq C_2 r^{-d} f\left(\tfrac{1}{r}\right),\quad  \quad r\in \left[\tfrac{1}{2},\infty\right). 
	\end{equation}
	Furthermore, for every $T>0$ we have
	\begin{equation}\label{h_and_H}
	  h(t) \approx H(t), \quad t\in (0,T),
	\end{equation}
  and there exist $C_3,C_4$ such that
  \begin{equation}\label{eq:doubling_h}
    C_3 \left(\frac{R}{r}\right)^{\frac{1}{\alpha_2}} 
		\leq \frac{h(R)}{h(r)} 
		\leq C_4 \left(\frac{R}{r}\right)^{\frac{1}{\alpha_1}}, \quad T \geq R\geq r > 0. 
	\end{equation}
\end{lem}
\begin{proof}
  \eqref{asympt_Psi} follows directly from \cite[Lemma 4.4]{Sztonyk2017}. We need to prove 
	\eqref{h_and_H} and \eqref{eq:doubling_h}. Let $g(s)=s^d f(s)$, $s\in (0,2]$.
  It follows from \eqref{eq:doubling_kappa} that
	$$
	   c_1\left(\frac{a}{b}\right)^{\alpha_2} \leq \frac{g(b)}{g(a)} \leq 
		 c_2 \left(\frac{a}{b}\right)^{\alpha_1}, \quad 2 \geq b \geq a >0.
	$$
	Taking $b=H(R)=g^{-1}(1/R)$ and $a=H(r)=g^{-1}(1/r)$ 
	in the above inequality we obtain
	\begin{equation}\label{eq:doubling_g}
	  c_1^{\frac{1}{\alpha_2}} \left(\frac{R}{r}\right)^{\frac{1}{\alpha_2}} 
		\leq \frac{H(R)}{H(r)}
		\leq c_2^{\frac{1}{\alpha_1}} \left(\frac{R}{r}\right)^{\frac{1}{\alpha_1}}, \quad r\leq R \leq \frac{1}{g(2)}.
	\end{equation}
	From \eqref{asympt_Psi} we get
  $$
	  \frac{1}{g^{-1}(s/C_2)} \leq \Psi_-(s) \leq \frac{1}{g^{-1}(s/C_1)},\quad 
		s \geq \max\{\Psi(1/2),C_2g(2)\},
	$$
	and since $h(t)=\frac{1}{\Psi_-(1/t)}$, the lemma follows
	from this, \eqref{eq:doubling_g} and the monotonicity of $h$ and $H$.
\end{proof}

\begin{lem}\label{integratingh} Let (B) hold and $a\in\R$, $b\geq 0$. 
  If $a-\frac{b}{\alpha_1}+1>0$
  then for every $T>0$ there exists $C_1=C_1(a,b,T)$ such that
  \begin{equation}\label{IL1}
	  \int_0^r s^a h(s)^{-b}\, ds \leq C_1 r^{a+1}h(r)^{-b}, \quad r\in (0,T),
	\end{equation}
	and if $a-\frac{b}{\alpha_2}+1<0$ then there exists $C_2=C_2(a,b)$ such that
	\begin{equation}\label{IU1}
	  \int_r^T s^a h(s)^{-b}\, ds \leq C_2 r^{a+1}h(r)^{-b}, \quad r\in (0,T).
	\end{equation}
	Furthermore,

	\begin{equation}\label{IL2}
	  \int_0^{1/\Psi(1/s)} u^a h(u)^{-b}\, du \leq C_1 \Psi(1/s)^{-a-1}s^{-b}, \quad s \in (0,h(T)),
	\end{equation}
	if $a-\frac{b}{\alpha_1}+1>0$ and
	\begin{equation}\label{IU2}
	  \int_{1/\Psi(1/s)}^T u^a h(u)^{-b}\, du \leq C_2 \Psi(1/s)^{-a-1}s^{-b}, \quad s\in (0,h(T)),
	\end{equation}
	if $a-\frac{b}{\alpha_2}+1<0.$
\end{lem}
\begin{proof}
  From Lemma \ref{lem:doubling_h} we get
  \begin{align*}
	  \int_0^r s^a h(s)^{-b}\, ds
		& \leq  c_1^b h(r)^{-b} r^{\frac{b}{\alpha_1}} \int_0^r s^{a-\frac{b}{\alpha_1}}\, ds 
		  \leq  c_1^b h(r)^{-b} r^{\frac{b}{\alpha_1}} 
		         \tfrac{1}{a-\frac{b}{\alpha_1}+1} r^{a-\tfrac{b}{\alpha_1}+1} \\
		&  =    c_1^b \tfrac{1}{a-\frac{b}{\alpha_1}+1} r^{a+1} h(r)^{-b},  
	\end{align*}
	and the first inequality follows. Similarly, Lemma \ref{lem:doubling_h} yields also
	the second inequality
	\begin{align*}
	  \int_r^T s^a h(s)^{-b}\, ds
		& \leq  c_2^{-b} h(r)^{-b} r^{\frac{b}{\alpha_2}} \int_r^T s^{a-\frac{b}{\alpha_2}}\, ds \\
		& \leq  c_2^{-b} \tfrac{1}{\frac{b}{\alpha_2}-a-1} r^{\frac{b}{\alpha_2}} h(r)^{-b} r^{a-\frac{b}{\alpha_2}+1}.
	\end{align*}
	Using \eqref{IL1} and \eqref{IU1} with $r=1/\Psi(1/s)+\epsilon,$
	we obtain upper estimates of both integrals by 
	$C_i (1/\Psi(1/s)+\epsilon)^{a+1}h(1/\Psi(1/s)+\epsilon)^{-b},$ $i=1,2.$
	Then, using \eqref{OnPsi_}, we get
	$$
	  \frac{1}{h\left(\tfrac{1}{\Psi(1/s)}+\epsilon\right)} = \Psi_-\left(\tfrac{1}{1/\Psi(1/s)+\epsilon}\right) \leq 1/s,
	$$
	for every $\epsilon$ sufficiently small and \eqref{IL2} and \eqref{IU2} follow.
\end{proof}

Recall that by $p(t,x)$ we denote the density of $\mu_t$ for $t>0$. 
The sharp estimates of $p$ in small times are given in \cite{KaletaSztonyk2017}. We
quote them in the following Lemma. 
\begin{lem}\label{p_t_est} 
  If (A) and (B) hold then for every $t_0>0$ there exist $C_1,C_2,C_3,C_4$ such that
	\begin{equation}\label{est_p_small_x}
	  C_1 h(t)^{-d} \leq p(t,x) \leq C_2 h(t)^{-d},\quad |x| \leq h(t), t\leq t_0,
	\end{equation}
	and
	\begin{equation}\label{est_p_large_x}
    C_3 tf(|x|) \leq p(t,x) \leq C_4 tf(|x|),\quad |x| \geq h(t), t\leq t_0.
  \end{equation}
  There exist constants $C_5,C_6$ such that
  \begin{equation}\label{eq:ptx_est}
    C_5 \left( h(t)^{-d} \wedge tf(|x|)\right) \leq p(t,x) \leq C_6 \left(h(t)^{-d} \wedge tf(|x|)\right),
  \end{equation}
  for $t\in (0,t_0), x\in\Rd$.
\end{lem}
\begin{proof}
  It follows from  the proof of  \cite[Theorem 1]{KaletaSztonyk2017} and \eqref{asympt_Psi}
that \eqref{est_p_small_x} and \eqref{est_p_large_x} hold.
The both estimates yield that $p(t,x)\geq (C_1 \wedge C_3) (h(t)^{-d} \wedge tf(|x|)).$
It follows also from \cite[Lemma 5]{KaletaSztonyk2017} and \cite[Lemma 6]{KaletaSztonyk2015} that $p(t,x)\leq c_1 h(t)^{-d}$
for all $x\in\Rd$ and $t\leq t_0,$ hence $p(t,x) \leq (c_1 \vee C_4) (h(t)^{-d} \wedge tf(|x|))$
for $|x| \geq h(t),$ $t\leq t_0.$ For $|x| \leq h(t)$ and $t\leq t_0$ we note that
\eqref{asympt_Psi} with $r=1/h(t)$ yields in particular that
	\begin{equation}\label{f(h)}
	  f(h(t)) \approx t^{-1}(h(t))^{-d}, \quad t\in (0,\Psi(1/2)),
	\end{equation}
hence 
$$
  tf(|x|) \geq c_2 t f(h(t)) \geq c_3 h(t)^{-d},
$$
and we get \eqref{eq:ptx_est} for all $x\in\Rd$ and $t\in (0,t_0].$
\end{proof} 

We will also use the following estimates of derivatives of transition densities. 
Such estimates are given in \cite{KaletaSztonyk2015} under slightly different assumptions. Using
auxiliary results of \cite{KaletaSztonyk2015} and \cite{KaletaSztonyk2017} we can easily prove them
in the current setting.

\begin{lem}\label{pderest}
If (A) and (B) hold then for every $t_0>0,$ $n\in\N$, $n>d,$ and $\beta\in\N^d$ there exists $C=C(n)$ such that
\begin{equation}\label{eq:pder}
  |\partial^\beta_x p(t,x)| \leq C h(t)^{-|\beta|}\min\left\{ h(t)^{-d}, tf\left(\tfrac{|x|}{2}\right)+ h(t)^{-d}\left(1+\tfrac{|x|}{h(t)}\right)^{-n}\right\},
\end{equation}
for all $t\leq t_0,x\in\Rd.$
\end{lem}
\begin{proof}
  Following the proofs in \cite{KaletaSztonyk2015} and \cite{KaletaSztonyk2017}, we consider the probability measures
	$\bar{\mu}_t$ and $\mathring{\mu}_t$ such that
	$$   
	  \Fourier{\mathring{\mu}_t}(u) = \exp\left(-t\int_{|y|<h(t)} \left(1-\cos(\scalp{u}{y})\right)\nu(dy)\right),
	$$
	$$
	  \Fourier{\bar{\mu}_t}(u) = \exp\left(-t\int_{|y|\geq h(t)} \left(1-\cos(\scalp{u}{y})\right)\nu(dy)\right),
		\quad u\in\Rd.
	$$
	We have of course $\mu_t = \mathring{\mu}_t\ast \bar{\mu}_t$ and for every $t>0$ the measure 
	$\mathring{\mu}_t$ is absolutely continuous with densities $\mathring{p}_t,$
	since $\Fourier{\mathring{\mu}_t}(u)\leq e^{-t\Phi(u)}e^{2t\nu(B(0,h(t))^c)},$ for every $t>0$.
	It follows from \cite[Lemma 4]{KaletaSztonyk2017} that for every $t_0>0$ there exists $c_1$ such that
	\begin{align} \label{eq:mu_by_f}
	  \bar{\mu}_t(A) \leq c_1 t f(\dist(A,0))(\diam(A))^d,\quad t\in (0,t_0],
	\end{align}
	for every bounded Borel set $A$ such that $\dist(A,0)\geq 3.$ 
	It follows from \cite[Lemma 3.2]{GrzywnySzczypkowski2021} that the assumptions of \cite[Lemma 12]{KaletaSztonyk2015}
	are satisfied; hence $\mathring{p}_t \in C^\infty_b(\Rd)$ and for every $n\in\N$ and every
	$\beta\in\N^d$ there exists a constant $c_2=c_2(n,|\beta|)$ such that 
	$$
	  |\partial^\beta_y \mathring{p}_t(y)| \leq c_2 [h(t)]^{-d-|\beta|}\left(1+\tfrac{|y|}{h(t)}\right)^{-n},\quad y\in\Rd,t\in (0,t_0].
	$$
	Since $p(t,x) = \int \mathring{p}_t(x-y) \bar{\mu}_t(dy),$ for $|x|>6$ we get
	\begin{align*}
	  |\partial^\beta_x & p_t(x)| 
		 =      \left|\int \partial^\beta_x \mathring{p}_t(x-y)\, \bar{\mu}_t(dy)\right| 
		 \leq  \int |\partial^\beta_x \mathring{p}_t(x-y)|\, \bar{\mu}_t(dy) \\
		& \leq  c_2 [h(t)]^{-d-|\beta|} \int \left(1+\tfrac{|x-y|}{h(t)}\right)^{-n}\, \bar{\mu}_t(dy) 
		    = c_2 [h(t)]^{-d-|\beta|} \int \int_0^{\left(1+\tfrac{|x-y|}{h(t)}\right)^{-n}}\, ds\bar{\mu}_t(dy) \\
		&  =    c_2 [h(t)]^{-d-|\beta|} \int_0^\infty \bar{\mu}_t\left(\left\{y\in\Rd :\: \left(1+\tfrac{|x-y|}{h(t)}\right)^{-n} > s\right\}\right)  ds \\
		&  =    c_2 [h(t)]^{-d-|\beta|} \int_0^1 \bar{\mu}_t\left(\left\{y\in\Rd :\: |x-y| < h(t)(s^{-1/n}-1) \right\}\right)  ds.
	\end{align*}
It follows from \eqref{eq:mu_by_f}, that this is bounded above by
  \begin{align*}
		&       c_1c_2 [h(t)]^{-d-|\beta|} \bigg( \int_{\left(1+\frac{|x|}{2h(t)}\right)^{-n}}^1 t f\left(|x|-h(t)(s^{-1/n}-1)\right) (2h(t)(s^{-1/n}-1))^d\, ds \\
		&       + \int_0^{\left(1+\frac{|x|}{2h(t)}\right)^{-n}} \, ds \bigg) \\
		& \leq  c_1c_2 [h(t)]^{-d-|\beta|} \left( 2^d t h(t)^d f\left(\tfrac{|x|}{2}\right)\int_0^1 s^{-d/n}\, ds
		         + \left(1+\frac{|x|}{2h(t)}\right)^{-n} \right) \\
		& \leq  c_3 [h(t)]^{-d-|\beta|} \left( t h(t)^d f\left(\tfrac{|x|}{2}\right)
		         + \left(1+\frac{|x|}{h(t)}\right)^{-n} \right).
	\end{align*}
	Now it follows from (29) in \cite{KaletaSztonyk2015} that there exists a constant $c_4$ such that
	$|\partial^\beta_x p_t(x)| \leq c_4 (h(t))^{-d-|\beta|}$ for every $x\in\Rd$ and $t\in (0,t_0].$ Hence
	\eqref{eq:pder} hold for all $|x|\geq 6$ and $t\in (0,t_0].$ Let $f_0(r):= f(r) \vee f(6).$
	Since $f(r)\leq f_0(r)$ for $r>0$ and $f_0(2r) \geq c_5 f_0(r)$ for some constant $c_5$ and every $r>0,$
	the assumptions of \cite[Theorem 3]{KaletaSztonyk2015} are satisfied with the profile function $f_0$
	which yields \eqref{eq:pder} for $|x|\leq 6.$
\end{proof}

In what follows we fix $t_0$ such that inequalities in Lemma \ref{p_t_est} and Lemma \ref{pderest} hold.


\section{Kato class and potential kernel}

We discuss here the equivalence of the different definitions of Kato classes. Note that
the case $a=1$ follows also from the results of \cite{GrzywnySzczypkowski2017}.

For $t>0$ and $a\geq 1$ we define 
$$
  V_{t}^a(x) = \left(\int_0^{t} p(s,x)^a\, ds\right)^{1/a}.
$$
It follows easily from \eqref{est_p_large_x} that
\begin{equation*}
  V_t^a(x) \approx t^{1+1/a} f(|x|),\quad |x|\geq h(t), t\in (0,t_0].
\end{equation*}

\begin{lem}\label{V_small_x} Let $a\geq 1$ and let (A) and (B) hold. 
Then
  \begin{displaymath}
	  V_t^a(x) \approx \frac{1}{|x|^d \Psi(1/|x|)^{1/a}}
		\approx \frac{1}{|x|^{d(1+1/a)} f(|x|)^{1/a}},  \quad |x|\leq h(t), t\in (0,t_0].
	\end{displaymath}
\end{lem}
\begin{proof}
Using (B), \eqref{asympt_Psi} and Lemma \ref{integratingh} with $T=t_0$, for $|x|\leq h(t),$ $t\leq t_0,$ we obtain
  \begin{align*}
	  V_t^a(x) 
		& \approx  \left(\int_0^{\frac{1}{\Psi(1/|x|)}} s^a f(|x|)^a \, ds + \int_{\frac{1}{\Psi(1/|x|)}}^t h(s)^{-ad} \, ds\right)^{1/a} \\
		& \leq  c_1 \left(\frac{f(|x|)^a}{\Psi^2(1/|x|)^{1+a}} + \frac{1}{\Psi(1/|x|)} \frac{1}{|x|^{ad}}\right)^{1/a} \\
		& \leq  c_2 \frac{1}{|x|^{d}\Psi(1/|x|)^{1/a}} \approx \frac{1}{|x|^{d(1+1/a)} f(|x|)^{1/a}}
	\end{align*}
	and 
	\begin{equation*}
	  V_t^a(x) 
		 \geq c_3 \left(\int_0^{\frac{1}{\Psi(1/|x|)}} s^a f(|x|)^a \, ds \right)^{1/a} 
		 \geq c_4 f(|x|) \frac{1}{\Psi(1/|x|)^{1+1/a}} \geq c_5 \frac{1}{|x|^{d}\Psi(1/|x|)^{1/a}}.
	\end{equation*}
\end{proof}

For a function $q$ we define
$$
  I_r^a(q) = \sup_{x\in\Rd} \int_{|y-x|<r} V_{t_0}^a(y-x) |q(y)|\, dy,\quad r>0, a\geq 1.
$$
By the above lemma we have that under (A) and (B), for every $a\geq 1$, the following holds: 
    \begin{equation} q\in \KatoC_a \quad \text{if and only if} \quad \label{eq:KatoTG}\lim_{r\to0}I_r^a(q)=0.\end{equation}

\begin{lem}\label{conv_st_estimate} If (A) and (B) hold, 
then for 
every $\kappa\in(0,1]$ there exist 
 constants $C_1,C_2$ such that for all $t \in (0,t_0)$ and $r \in (0,h(t_0))$, we have
  \begin{equation}\label{conv_spacetime}
    \sup_{x\in\Rd}  \int_{\Rd} |q(z)|  \left(\int_0^t p(s,\kappa(z-x))^a \, ds \right)^{1/a} \, dz \leq  
		  C_1  \left(1 + C_2 \left(t\Psi(1/r)\right)^{1/a} \right) I_r^a(q).
	\end{equation}
\end{lem}
\begin{proof}
  Using \cite[Lemma 4.2]{BogdanSzczypkowski2014} for the radially decreasing functions
	$$k(x) = \left(\int_0^t \left( h(s)^{-d} \wedge s f(\kappa|x|) \right)^a\, ds\right)^{1/a} \quad  \text{and} \quad 
	K(x)=\left(\int_0^{t_0} \left( h(s)^{-d} \wedge s f(\kappa|x|) \right)^a\, ds\right)^{1/a}, $$ 
	we get
	$$
	  \sup_{x\in\Rd} \int_{\Rd} |q(z)| k(x-z)\, dz \leq c_3 \sup_{x\in\Rd} \int_{B(x,r)} |q(z)| K(x-z)\, dz
	$$
	for $r \in (0, h(t_0))$, where $c_3 = 1 + c_1/c_2$ and $c_1 = \int_{\Rd} k(x)\, dx$, $c_2(r)=K(r,0,...,0)|B(0,r/2)|$.
	It follows from Lemma \ref{p_t_est}, Lemma \ref{V_small_x}, \eqref{asympt_Psi} and (B) that
	\begin{displaymath}
		c_2 \geq c_4 V_{t_0}^a((\kappa r,0,...,0)) r^d \geq c_5 \Psi(1/\kappa r)^{-1/a} 
		\geq c_6 \Psi(1/r)^{-1/a}.  
	\end{displaymath}
	We observe that $c_1 = \kappa^{-d} \int_{\Rd} \left( \int_0^t \left( h(s)^{-d} \wedge s f(|x|) \right)^a \, ds\right)^{1/a}\, dx \approx \int_{\Rd} V_t^a(x)\, dx$, and
	\begin{align*}
	  \int_{\Rd} V_t^a(x)\, dx
		& = \int_{|x|\leq h(t)} V_t^a(x)\, dx + \int_{|x| >  h(t)} V_t^a(x)\, dx \\
		& = I + II.
	\end{align*}
	We have
	  \begin{align*}
		II
		& \leq c_7 t^{1+1/a}\int_{|x| > h(t)} f(|x|)\, dx,
	\end{align*}
	and from \eqref{Pineq} we get $II\leq c_8 t^{1+1/a}\Psi(1/h(t))=c_8 t^{1/a}.$
	
	Furthermore, it follows from Lemma \ref{integratingh} and Lemma \ref{lem:doubling_h} that
	\begin{align*}
	 I & \leq c_9 \int_{|x|\leq h(t)}\left(
	  \int_0^{\frac{1}{\Psi(1/|x|)}}s^a f(|x|)^a\, ds + \int_{\frac{1}{\Psi(1/|x|)}}^t h(s)^{-ad}\, ds\right)^{1/a}\, dx \\
		 & \leq c_{10} \int_{|x|\leq h(t)}\left( \frac{f(|x|)^a}{\Psi(1/|x|)^{a+1}} 
		    +  \frac{1}{|x|^{ad}\Psi(1/|x|)} \right)^{1/a}\, dx \\
		 & \leq c_{11}  \int_{|x|\leq h(t)} \frac{1}{|x|^{d}\Psi(1/|x|)^{1/a}}\, dx 
		  = c_{12} \int_0^{h(t)} u^{-1} \Psi(1/u)^{-1/a}\, du \\
		&	= c_{12} \int_0^{1} s^{-1} \Psi(1/(h(t)s))^{-1/a}\, ds.
	\end{align*}
	
	It follows from (B) that for every $A>1$ we have
	\begin{displaymath}
		\frac{f\left(\frac{1}{Ar}\right)}{f\left(\frac{1}{r}\right)} \geq c_{13} A^{\alpha_1+d}, \quad r\geq 1,
	\end{displaymath}
	hence, by Lemma \ref{lem:doubling_h} we get $\Psi(As) \geq c_{14} A^{\alpha_1} \Psi(s)$ for $s>1,$ and
	$\Psi(1/(h(t)s))\geq c_{14} s^{-\alpha_1}\Psi(1/h(t)) = c_{14} s^{-\alpha_1}t^{-1}.$ Hence,
	\begin{displaymath}
		I \leq c_{15} t^{1/a} \int_0^1 s^{-1+\alpha_1/a}\, ds = c_{16} t^{1/a}. 
	\end{displaymath}
	
	Finally, we note that for $|x|\leq  h(t_0)$ we have $K(x) \approx V_{t_0}^a(\kappa x) \leq c_{17} V_{t_0}^a(|x|)$ and the lemma follows.
	
\end{proof}

\begin{cor}\label{CorKato_Eq} If (A) and (B) hold 
  then
  $q\in \KatoC_a,$ for $a\geq 1$ if and only if
	\begin{equation}\label{Kato2}
	  \lim_{t \downarrow 0} \sup_{x\in\Rd} \int_{\Rd} |q(y)| 
	  \left(\int_0^t \left[p(s,x-y)\right]^a\, ds \right)^{\frac{1}{a}}\, dy = 0.
  \end{equation}

\end{cor}
\begin{proof}
  It follows from Lemma \ref{V_small_x} that $V_t^a(x)\approx V_{t_0}^a(x)$ for every $t\in (0,t_0)$ and 
	$|x|\leq \tfrac{1}{2} h(t)$.
	Hence 
\begin{align*}
  \int_{\Rd} \left(\int_0^t \left[ p(s,x-y) \right]^a \, ds\right)^{1/a} |q(y)| \,dy 
  &	=  \int_{\Rd} V_{t}^a(x-y) |q(y)|\, dy \\
	& \geq  c_1 \int_{|x-y|\leq \tfrac{1}{2} h(t)} V_{t_0}^a (x-y)|q(y)|\, dy.
\end{align*}
Let $\varepsilon>0$. The above estimate and \eqref{Kato2} yield that there exists 
$t_1>0$ and $\delta_1=\frac{1}{2}h(t_1)>0$ such that
\begin{eqnarray*}
  \sup_{x\in\Rd} \int_{|x-y|\leq \delta} V_{t_0}^a (x-y)|q(y)|\, dy 
	& \leq & c_2 \sup_{x\in\Rd} \int_{\Rd} \left(\int_0^{t_1} \left[ p(s,x-y) \right]^a \, ds \right)^{1/a} |q(y)|\, dy\leq \varepsilon,
\end{eqnarray*}
for all $\delta \in (0,\delta_1)$. Using Lemma \ref{V_small_x} we obtain \eqref{extraKato}. We get the opposite implication
by using \eqref{conv_spacetime} with $r=h(t).$

\end{proof}

\begin{cor}\label{Cor_qint} If (A) and (B) hold and $q\in\KatoC$ 
  then 
	\begin{equation}\label{int_q_lim}
    \lim_{r\to 0} \sup_{x\in\Rd} \int_{|x-y|<r} |q(y)|\, dy = 0,
	\end{equation} 
	and for every $r_0>0$ we have
  \begin{equation}\label{int_q_bd}
	  \sup_{x\in\Rd} \int_{|x-y|<r_0} |q(y)|\, dy < \infty.
	\end{equation}
\end{cor}
\begin{proof}
  Using Lemma \ref{V_small_x} 
  and \eqref{eq:doubling_kappa} 
	for $r\leq h(t_0) \wedge (2r_0)$ we get
	\begin{align*}
	  \inf_{y\in B(0,r)} V_{t_0} (y)
		& 
		\geq  \inf_{y\in B(0,r)} \frac{c_2}{|y|^{2d}f(|y|)} \\
		& \geq  \inf_{y\in B(0,r)} \frac{c_3}{|y|^{d-\alpha_2}} 
		\geq  c_3 r^{-d+\alpha_2},
	\end{align*}
	hence
	\begin{align*}
	  \int_{|x-y|<r} |q(y)|\, dy 
		& \leq  \left(\inf_{y\in B(0,r)} V_{t_0} (y)\right)^{-1}  \int_{|x-y|<r} V_{t_0}(x-y) |q(y)|\, dy \\
		& \leq  c_4 r^{d-\alpha_2} \int_{|x-y|<r} V_{t_0}(x-y) |q(y)|\, dy,
	\end{align*}
	and \eqref{int_q_lim} follows from \eqref{eq:KatoTG}. It follows from \eqref{int_q_lim}
	that there exists $\delta_0>0$ such that 
	$$
	  \sup_{x\in\Rd} \int_{|x-y|<r} |q(y)|\, dy < 1, 
	$$
	for $r\leq\delta_0$, which yields 
	\eqref{int_q_bd} for all $r_0\leq \delta_0$. For $r_0>\delta_0$ we get \eqref{int_q_bd} by
	covering $B(x,r_0)$ by finite number (independent of $x$) of balls with radius $\delta_0$. 
	
\end{proof}

\section{Perturbation} \label{sec:perturbation}

\begin{lem}\label{lem:flocal} If (A) and (B) hold then for every $R>0$ there exists a constant $C=C(R)$ such that
  \begin{equation}\label{flocal_a}
	  f(r) \leq C f(r+s),\quad r>0,\, 0 < s \leq R \wedge \tfrac{r}{2}.
  \end{equation}
\end{lem}
\begin{proof}  
  If $r\leq 1$ then for $0<s\leq \frac{r}{2} < 1$ from \eqref{eq:doubling_kappa} we have
	$$
	  \frac{f(r)}{f(r+s)} \leq c_1 \left(\frac{r+s}{r}\right)^{d+\alpha_2} \leq c_1 \left( \frac{3}{2}\right)^{d+\alpha_2}.
	$$
  It follows from \cite[Lemmas 1 and 3]{KaletaSztonyk2017} that for every $R\geq 1$ there exists a constant $c_0$ such that we have $f(r+s-R) \leq c_0 f(r+s) $ provided
	$r+s \geq 3 R,$ and since $f$ is nonincreasing we obtain $f(r) \leq c_0 f(r+s)$ for $s\leq R,$ $r+s \geq 3R$ and \eqref{flocal_a}
	for $r\geq 3R.$ 
	For $r\in (1,3R)$ and $s\leq \frac{r}{2}$ we have $f(r) \leq f(1) = \frac{f(1)}{f(9R/2)} f(9R/2) \leq \frac{f(1)}{f(9R/2)} f(r+s).$
	This yields \eqref{flocal_a} for every $R\geq r_0$ with $C=\max\{c_0,M_2(3/2)^{d+\alpha},f(1)/f(9R/2)\}.$ 
	We note that if \eqref{flocal_a} holds for some $R>0$, (e.g., for $R=1$) then it holds with the same constant for every smaller $R.$
		
\end{proof}

We note that the condition $\alpha_2<d$ in (B) is not necessary also in the following collorary .

\begin{cor}\label{pleqp}
  There exists a constant $C$ such that
	\begin{displaymath}
		p(t,x+y) \leq C p(t,x), \quad t\in (0,t_0],\, x\in\Rd,\, |y|\leq \frac{1}{2}h(t).
	\end{displaymath}
\end{cor}
\begin{proof}
  If $|x|\geq \frac{3}{2}h(t)$ then using Lemma \ref{p_t_est} and Lemma \ref{lem:flocal} we get
	\begin{displaymath}
		p(t,x+y) \leq c_1 t f(|x+y|) \leq c_1 t f(|x|-\tfrac{1}{2}h(t)) \leq c_2 t f(|x|) \leq c_3 p(t,x).
	\end{displaymath}
	For $|x|\leq \frac{3}{2}h(t),$ using $f(\frac{3}{2} h(t)) \approx f(h(t)) \approx t^{-1} h(t)^{-d}$ (see \eqref{f(h)}), we get
	\begin{displaymath}
		p(t,x+y) \leq c_4 h(t)^{-d}
	\end{displaymath}
	and
	\begin{displaymath}
		p(t,x) \geq c_5 (h(t)^{-d} \wedge tf(|x|) ) \geq c_5 \left(h(t)^{-d} \wedge tf\left(\tfrac{3}{2}h(t)\right) \right) \geq c_6 h(t)^{-d} \geq c_7 p(t,x+y).
	\end{displaymath}
\end{proof}

\begin{lem}\label{int_gq_est}
  If (A) and (B) hold 
	and $q\in {\cal J},$  then for every $r>0$
  there exists $C=C(r)$ such that 
          $$ 
            \int_{|y-z|\geq r, \atop |x-z| \geq r } f(|z-x|)|q(z)|f(|y-z|) \, dz \leq 
						C M_q f(|y-x|), \quad 
						|y-x|\geq 6 r,
          $$
					where $M_q=\sup_{x\in\Rd} \int_{|z-x|\leq 3r} |q(z)|\, dz$.
\end{lem}
\begin{proof}
  If the inequality holds for some $r_1>0$ then it holds also for every $r>r_1$ with the same constant so 
	it is enough to prove the inequality for $r\in (0,1].$
	
	First we prove that for every $r\in (0,1]$ there 
	exist a constant $c_r>0$ such that 
  \begin{equation}\label{convcond_ext}
    \nu_{r} * \nu_{r}(x) \leq c_r \, \nu(x), \quad |x| \geq r.
  \end{equation}
	Indeed, it follows from (A) and Lemma \ref{lem:flocal} that for $|x|\geq 3$ we have
	\begin{align*}
    \nu_{r} * \nu_{r}(x)
		& \leq c_1 \int_{ \{1 \geq |y|\geq r\} \atop \cup\{ 1\geq |x-y| \geq r\} } f(|x-y|) f(|y|) \, dy 
		  + c_2 \nu(x) \\
		& \leq c_3 \left( f(|x|-1)f(r) + \nu(x)\right) \leq c_4 \nu(x).
	\end{align*}
  Now we observe that for $r\leq |x|\leq 3$ we have 
	$\tfrac{\nu_{r} * \nu_{r}(x)}{\nu(x)} \leq \tfrac{c_5f(r)\nu(B(0,r)^c)}{f(3)},$ which
	yields \eqref{convcond_ext}.
	
	Fix $r \in (0,1]$ and $x,y \in \R^d$ such that $|x-y| \geq 6r$.
  Let $\R^d = \bigcup_{j\in\N} K_j$ be a decomposition of $\Rd$ into cubes with diameter
	$r$.
	Let $I_1=\{j\in\N :\: K_j\cap B(x,2r) = \emptyset, K_j\cap B(y,2r) = \emptyset \}$, 
	$I_2= \N\setminus I_1$.
	Let $\theta_j$ denote the center of $K_j$.
	Using Lemma \ref{lem:flocal} with $R=r$, Corollary \ref{Cor_qint} and \eqref{convcond_ext} we obtain 
	\begin{align*}
	  \sum_{j\in I_1} \int_{K_j} f(|z-x|)|q(z)|f(|y-z|)  \, dz
		& \leq  \sum_{j\in I_1} c_6 f(|\theta_j-x|)f(|y-\theta_j|) \int_{K_j}  |q(z)| \, dz \\
		& \leq  M_q c_6 \sum_{j\in I_1}  f(|\theta_j-x|)f(|y-\theta_j|) \\
		& \leq  M_q c_6^2 \left(\tfrac{\sqrt{d}}{r}\right)^d \sum_{j\in I_1} \int_{K_j} f(|z-x|)f(|y-z|)\, dz \\
		& \leq  M_q c_6^2 \left(\tfrac{\sqrt{d}}{r}\right)^d \int_{|z-x|>r,|y-z|>r} f(|z-x|)f(|y-z|)\, dz \\
		& \leq  c_7 M_q f(|y-x|).
	\end{align*}
	Furthermore, using Lemma \ref{lem:flocal} again (with $R=3r$) we get
	\begin{align*}
	 & \sum_{j\in I_2} \int_{K_j \cap B(x,r)^c \cap B(y,r)^c} f(|z-x|) |q(z)| f(|y-z|) \, dz  
					& \\
		\leq  & \int_{r \leq |z-x|\leq 3r} f(|z-x|) |q(z)| f(|y-z|) \, dz 
		      +  \int_{r \leq |y-z|\leq 3r} f(|z-x|) |q(z)| f(|y-z|) \, dz  & \\
		\leq & f(r) \int_{r \leq |z-x|\leq 3r}  |q(z)| f(|y-z|) \, dz  
	       + f(r) \int_{r \leq |y-z|\leq 3r} f(|z-x|) |q(z)|  \, dz & \\
		\leq & c_8 f(r) f(|y-x|)\int_{|z-x|\leq 3r} |q(z)| \, dz 
		      + c_8 f(r) f(|y-x|)\int_{|y-z|\leq 3r} |q(z)|  \, dz & \\
	  \leq & 2 c_8 f(r) M_q f(|y-x|)
	\end{align*}
	and the lemma follows.
	
\end{proof}

We prove below a restricted version of 3G theorem for the transition function $p(t,x)$.

\begin{lem}\label{lem:3G} If (A) and (B) hold then for every $R\geq 1$ 
there exists a constant $C=C(R)$ such that
	$$
		p(t,x) \wedge p(s,y) \leq C p(t+s,x+y),\quad t,s>0,\, t+s \leq t_0,\, x,y\in B(0,R).
	$$
\end{lem}
\begin{proof}
Using \eqref{eq:ptx_est} and Lemma \ref{lem:doubling_h} we get
  \begin{align*}
	  p(t,x) \wedge p(s,y) 
		& \leq c_1 \left[ h(t)^{-d}\wedge (t f(|x|) )\right] \wedge \left[h(s)^{-d} \wedge (sf(|y|))\right] \\
		& \leq c_1\, h(t\vee s)^{-d} \wedge \left[ (t\vee s) f(|x|\vee |y|)\right] \\
		& \leq c_1\, h\left(\tfrac{t+s}{2}\right)^{-d} \wedge \left((t+s)f\left(\tfrac{|x|+|y|}{2}\right)\right) \\
		& \leq c_1\left(c_2 2^{\frac{1}{\alpha_1}}\right)^d\, h\left(t+s\right)^{-d} \wedge \left((t+s)f\left(\tfrac{|x|+|y|}{2}\right)\right)
	\end{align*}
	If $x,y\in B(0,2)$ then using 
	\eqref{eq:doubling_kappa} 
	and the monotonicity of $f$ 
	we obtain
  \begin{displaymath}
	  f\left(\tfrac{|x|+|y|}{2}\right) \leq
		c_3 f(|x|+|y|) \leq c_3 f(|x+y|),
  \end{displaymath}	
	and the lemma follows from \eqref{eq:ptx_est}. If $1\leq |x| \leq R,$ or $1\leq |y| \leq R,$
	then 
	$$
	  f\left(\tfrac{|x|+|y|}{2}\right) \leq f\left(\tfrac{1}{2}\right) \leq \tfrac{f\left(\tfrac{1}{2}\right)}{f(2R)} f(|x+y|).
	$$
\end{proof}
We note that if $f$ has doubling property, i.e., $f(s)\leq c f(2s)$ for all $s>0$ 
then the above 3G inequality holds for all $x,y\in\Rd$ and the following theorem can be proved easily.

\begin{thm} \label{th:VeryIT}
  If (A) and (B) hold 
	and $q\in\KatoC,$ then there exists $\beta$
	such that for every $\eta\in (0,1)$
	there exists $t_\eta>0$ such that 
  \begin{equation}\label{BCond}
    \int_0^t \int_{\Rd} p(s,z-x) |q(z)| p(t-s,y-z)\, dzds \leq  (\eta+\beta t)  \, p(t,y-x), \
  \end{equation}
	for all $ x,y\in\Rd,\,  t\in (0,t_\eta].$ 
\end{thm}
\begin{proof}
	We consider first the case $|x-y| > 6 h(t_0).$
	Denote the above double integral by $J(x,y,t):= J$. We have
	\begin{align*}
	  J 
		& =   \int_0^t \int_{|z-x|> h(t_0), \atop |z-y|> h(t_0)} 
		      + \int_0^t \int_{|y-z|\leq h(t_0)}
		      + \int_0^t \int_{|z-x|\leq h(t_0)} \\
		& =:  J_1 + J_2 + J_3, 
	\end{align*}
	and note that by symmetry of $p$ we have $J_2(x,y)=J_3(y,x).$
	Using Lemma \ref{p_t_est} and Lemma \ref{int_gq_est} for every $t\in (0,t_0]$ we obtain
	\begin{align*}
	  J_1 
		&   =   \int_0^t \int_{|z-x|>h(t_0), \atop |z-y|> h(t_0)} p(s,z-x) |q(z)| p(t-s,y-z)\, dzds \\
		& \leq  c_1 \int_0^t s(t-s)\, ds \cdot 
		        \int_{|z-x|> h(t_0), \atop |z-y|>h(t_0)} f(|z-x|)|q(z)|f(|y-z|)\, dz \\
		&   =   c_1 \frac{t^2}{6} t \int_{|z-x|> h(t_0), \atop |z-y|>h(t_0)} 
		         f(|z-x|)|q(z)|f(|y-z|) \, dz \\
		& \leq  c_2 \frac{t^2}{6} t f(|y-x|) \leq c_3 t_0 t p(t,y-x).
	\end{align*}
	From Lemma \ref{lem:flocal} with $R= h(t_0)$ and Lemma \ref{conv_st_estimate} 
	for every $r\in (0,t_0]$ we get
	 \begin{align*}
	  J_2(x,y) = J_3(y,x) 
		&   =   \int_0^t \int_{|z-y|\leq h(t_0)} p(s,z-x) |q(z)| p(t-s,y-z)\, dzds \\
		& \leq  c_4 t \int_0^t \int_{|z-y|\leq h(t_0)} f(|x-z|)|q(z)|p(t-s,y-z)\, dzds \\
		& \leq  c_5 t f(|y-x|) \int_0^t \int_{\Rd} |q(z)|p(s,y-z)\, dz ds\\
		& \leq  c_6 (1+c_7 t\Psi(\tfrac{1}{r}))I_r(q) p(t,y-x).
	\end{align*}
	
	Let now $|x-y|<6h(t_0)$. We use the same notation $J$ as above. We have
	  \begin{eqnarray*}
	  J 
		& \leq  & \int_0^t \int_{|x-z|< h(t_0), \atop |z-y|< h(t_0)} 
		      + \int_0^t \int_{|x-z|\geq  h(t_0)}
		      + \int_0^t \int_{|z-y|\geq h(t_0)}  \\
		& =: & I_1 + I_2 + I_3,
	\end{eqnarray*}
	and note that $I_2(x,y)=I_3(y,x).$
	From Lemma \ref{lem:3G} and Lemma \ref{conv_st_estimate} we get
  \begin{align*}
	  I_1 
		&   =   \int_0^t \int_{|z-x| <  h(t_0), \atop|y-z| < h(t_0)} 
		    p(s,z-x) |q(z)| p(t-s,y-z)\, dzds \\
		& \leq  \int_0^t c_8 \, p(t,y-x) \int_{|z-x| <  h(t_0), \atop |y-z| <  h(t_0)} 
		    |q(z)| (p(s,z-x)+  p(t-s,y-z))\, dzds \\
    & \leq  c_8 \, p(t,y-x) \int_0^t \int_{\Rd} |q(z)| (p(s,z-x)+  p(s,y-z))\, dz \\
		& \leq  c_9 (1+c_{10}\,t\Psi(\tfrac{1}{r}))\, I_r(q) p(t,y-x).
	\end{align*}
	
	Furthermore,
	\begin{align*}
	  I_2(x,y)=I_3(y,x) 
		&   =   \int_0^t \int_{|z-x|> h(t_0)} p(s,z-x) |q(z)| p(t-s,y-z)\, dzds \\
		& \leq  c_{11} t \int_0^t \int_{|z-x|>h(t_0)} f(|z-x|)|q(z)|p(t-s,y-z)\, dzds.
	\end{align*}
	Observe that $f(|z-x|)\leq c_{12} f(|y-x|)$ and $tf(|z-x|)\leq h(t)^{-d}$, since 
	$|z-x|> h(t_0)\geq h(t)$ (see \eqref{f(h)}). Hence, using Lemma \ref{p_t_est} and
	Lemma \ref{conv_st_estimate}
	again we get
	  \begin{eqnarray*}
		  I_2
			& \leq & c_{13} p(t,y-x) \int_0^t \int_{\Rd} |q(z)|p(t-s,y-z)\, dzds \\
			& \leq & c_{14} (1+c_{15}t\Psi(\tfrac{1}{r}))I_r(q) p(t,y-x).
		\end{eqnarray*}
	Now, taking $r=h(t)$ we get $t\Psi(\tfrac{1}{r})=1$ and the lemma follows 
	from the fact 
	that $\lim_{r\to 0^+}I_r(q)=0$ by \eqref{eq:KatoTG}.

\end{proof}

 It follows easily from \eqref{BCond} that $q$ is \textsl{relatively Kato} in 
the sense of \cite{BogdanHansenJakubowski2008}, i.e., for every $\eta>0$ there exists $h_{\eta}>0$ such that
\begin{equation}\label{eq:RelativelyKato}
  \int_0^t \int_{\Rd} p(s,z-x) |q(z)| p(t-s,y-z)\, dzds \leq \eta \, p(t,y-x),
\end{equation}
for $t\in (0,h_{\eta}).$
Consequently, Theorem 2 in \cite{BogdanHansenJakubowski2008}, and the fact that $p(t,x)>0$ (see \cite[Theorem 2]{KaletaSztonyk2015}) for all $t>0,x\in\Rd$ yield that there exists a unique transition density $\widetilde{p},$
such that 
\begin{equation}\label{pertp}
  \widetilde{p}(t,x,y) = p(t,y-x) + \int_0^t \int_{\Rd} p(u,z-x)q(z)\widetilde{p}(t-u,z,y)\, dz du,
\end{equation}
for $t>0,x,y\in\Rd.$
We have $\widetilde{p}(t,x,y)>0$ for all $t>0,x,y\in\Rd,$ and the Chapman-Kolmogorov equation
holds:
$$
  \int_{\Rd} \widetilde{p}(s,x,z)\widetilde{p}(t,z,y)\, dz = \widetilde{p}(s+t,x,y),\quad s,t>0, x,y\in\Rd.
$$
The density $\widetilde{p}$ is given by a series (see (10) and (23) in \cite{BogdanHansenJakubowski2008})
$$
  \widetilde{p}(t,x,y) = \sum_{n=0}^\infty p_n(t,x,y),
$$
where $p_0(t,x,y)=p(t,y-x),$ and
$$
  p_n(t,x,y) = \int_0^t \int_{\Rd} p_{n-1}(u,x,z)q(z)p(t-u,y-z) \, dzdu.
$$
Furthermore, it follows from 
  \cite[Lemma 5]{BogdanHansenJakubowski2008} that for every $\eta>0$ we have
\begin{equation}\label{eq:RelKato2}
  \int_0^t \int_{\Rd} p(s,z-x) |q(z)| p(t-s,y-z)\, dzds \leq \eta(1+(1/h_{\eta})t) \, p(t,y-x), 
\end{equation}
for all $t>0,\, x,y,\in\Rd,$
and by \cite[Theorem 3 and
(27)]{BogdanHansenJakubowski2008} for every natural number $n$ and $0<t<n h_\eta$ we get
\begin{equation*}
 (1-\eta)^n \leq \frac{\widetilde{p}(t,x,y)}{p(t,y-x)} \leq \frac{1}{1-\eta} \exp{\frac{\eta t}{h_{\eta}(1-\eta)}}, \quad x,y\in\Rd.
\end{equation*}

\section{Solution} \label{sec:solution}

The proof of the following lemma can be found (in more generall setting) in \cite[(1.16)]{BogdanButkoSzczypkowski2016}. 
\begin{lem}\label{FunSp}
For every $x\in\Rd,$ $s>0$ and $\phi\in C^\infty_c(\R\times \Rd)$ we have 
$$
  \int_0^\infty \int p(t,y-x)\left( \partial_t \phi(s+t,y) + \gener\phi(s+t,y) \right)\, dydt = -\phi(s,x).
$$
\end{lem}

\begin{prop}\label{FundS}  If $q\in {\cal J}$, (A) and (B) hold 
then
for every $x\in\Rd,$ $s>0$ and  $\phi\in C^\infty_c(\R\times \Rd)$ we have
\begin{equation*}
  \int_0^\infty \int_{\Rd} \widetilde{p}(t,x,y)\left( \partial_t \phi(s+t,y) + \gener\phi(s+t,y) + q(y)\phi(s+t,y) \right)\, dydt = -\phi(s,x). 
\end{equation*}
\end{prop}

\begin{proof}
This is a strightforward conclusion from \cite[Lemma 4]{BogdanJakubowskiSydor2012}. We need only
to verify the assumptions.
It follows from the general semigroup theory (see, e.g., \cite[Chapter 2]{BoettcherSchillingWang2013}) that
for every $\varphi\in C^2_\infty(\Rd)$ we have $\gener\varphi\in C_\infty(\Rd)$
and $$
  \|\gener \varphi\|_\infty \leq c \left(\int_{\Rd\setminus\{ 0\}} \left(1\wedge |y|^2\right)\nu(y)\, dy\right) \|\varphi\|_{(2)}, 
$$
for some constant $c,$ where
\begin{displaymath}
  \|\varphi\|_{(2)}:= \sum_{0\leq |\beta|\leq 2} \sup_{x\in\Rd} | \partial^{\beta} \varphi(x) |.	
\end{displaymath}

Therefore $|\partial_t \phi(t,y) + \gener \phi(t,y)|$  is a function bounded by a constant depending only on $d,$ $\phi$ and $\nu$. This fact, \eqref{eq:RelKato2} and Lemma \ref{FunSp} are sufficient for
Lemma 4 in \cite{BogdanJakubowskiSydor2012} to hold and hence the proposition follows.

\end{proof}

\begin{proof}[Proof of Theorem \ref{MTh}]
First part of the proof follows
directly from \eqref{tildep_est} and Proposition \ref{FundS}. We need to prove only the continuity of $\widetilde{p}.$
First, using induction we prove that the functions $p_n$ are continuous for every $n\in\N$ and $t\in (0,t_0)$.
We recall that
$$
  p_n(t,x,y)  = \int_0^t \int_{\Rd} p_{n-1}(u,x,z)p(t-u,y-z) q(z)\, dzdu.
$$
It follows from Lemma 7. in \cite{BogdanHansenJakubowski2008} (applied to $|q|$) that
\begin{equation}\label{pn_est}
  |p_n(t,x,y)| \leq p(t,y-x) \sum_{k=0}^n \binom{n}{k} \frac{(\beta t)^k}{k!} \eta^{n-k}, \quad t>0, x,y\in\Rd,
\end{equation}
where $\beta=\eta/h_\eta$ and $\eta,h_\eta$ are constants given in \eqref{BCond}.
Let $\varepsilon\in (0,t/2)$ and
$$
  f_{t,\varepsilon}(x,y) = \int_\varepsilon^{t-\varepsilon} \int_{\Rd} p_{n-1}(u,x,z)p(t-u,y-z) q(z)\, dzdu.
$$
Let $x_0,y_0\in\Rd$ and $(x_k,y_k) \to (x_0,y_0)$ as $k\to \infty.$
From continuity of $p_{n-1}(u,\cdot,z)$ and $p(t-u,\cdot)$ it follows that
$$
  p_{n-1}(u,x_k,z)p(t-u,y_k-z) q(z) \to p_{n-1}(u,x_0,z)p(t-u,y_0-z) q(z), 
$$
and using \eqref{pn_est} and Collorally \ref{pleqp} we get 
$$
  | p_{n-1}(u,x_k,z)p(t-u,y_k-z) q(z) | \leq c_1  p(u,z-x_0)p(t-u,y_0-z) |q(z)|
$$
for every $z\in\Rd$ and $|x_k-x_0|\leq \frac{1}{2}h(\varepsilon),$ $|y_k-y_0|\leq \frac{1}{2}h(\varepsilon).$
Hence, by the dominated convergence theorem, the function $f_{t,\varepsilon}(x,y)$ is continuous on $\Rd\times \Rd.$

Furthermore,
\begin{align*}
  |f_{t,\varepsilon}(x,y) - p_n(t,x,y)| 
	& \leq c_2 \int_0^\varepsilon \int_{\Rd} p(u,z-x)p(t-u,y-z) |q(z)|\, dzdu \\
	& + c_3 \int_{t-\varepsilon}^t \int_{\Rd} p(u,z-x)p(t-u,y-z) |q(z)|\, dzdu \\
	& \leq c_4 h(t/2)^{-d} \sup_{x\in\Rd} \int_0^\varepsilon \int_{\Rd} p(u,z-x) |q(z)|\, dzdu
\end{align*}
so it follows from Lemma \ref{conv_st_estimate} and Corollary \ref{CorKato_Eq} 
that $f_{t,\varepsilon}(x,y) \to p_n(t,x,y)$ uniformly
in $(x,y),$ as $\varepsilon\to 0,$ hence $p_n(t,\cdot,\cdot)$ is continuous.

Using \eqref{pn_est} again we see that the series defining $\widetilde{p}(t,x,y)$ converges uniformly hence
$\widetilde{p}(t,x,y)$ is jointly continuous for every $t\in (0,t_0).$

If $t\geq t_0$ then there exists $m\in\N$ such that $t/m\in (0,t_0).$ Using the Champan--Kolmogorov formula we get
$$
  \widetilde{p}(t,x,y) = \int ... \int \widetilde{p}(t/m,x,y_{m-1})...\widetilde{p}(t/m,y_2,y_1)\widetilde{p}(t/m,y_1,y)\, dy_{m-1}...dy_1
$$
and the joint continuity of $\widetilde{p}(t,\cdot,\cdot)$ follows from the continuity of $\widetilde{p}(t/m,\cdot,\cdot)$
and the inequalities
$$
  \widetilde{p}(t/m,\tilde{x},y) \leq c_5 \widetilde{p}(t/m,x,y), \quad |\tilde{x} -x| \leq \frac{1}{2}h(t/m),  
$$ 
$$
  \widetilde{p}(t/m,x,\tilde{y}) \leq c_6 \widetilde{p}(t/m,x,y), \quad |\tilde{y} -y| \leq \frac{1}{2} h(t/m),
$$
which follow from \eqref{tildep_est} and Corollary \ref{pleqp}.
\end{proof}


\section{Regularity} 

First we prove a version of 4G inequality for the heat kernels $p(t,x).$ For $a>0$ we denote $g_a(t,x)=a^d p(t,ax)$.  We note that Proposition \ref{prop:4G} and Lemmas \ref{Hoelderineq}- \ref{lem_DiffEst} hold also
in the case $\alpha_2>1=d.$ 

\begin{prop}\label{prop:4G}
Let $0<a<b<\infty$. If  (A), (B) and (C) hold  
 then there exists $C=C(a,b)$ such that
$$
  g_b(t,x) g_a(s,y)\leq C \left[g_{b-a}(t,x)\vee g_a(s,y)\right]g_a(t+s,x+y),
$$
for all $t,s>0,\, t+s\leq t_0$, $x,y\in\Rd$.
\end{prop}

\begin{proof}
First we note that for every $ 0< \rho < \eta $ we have
\begin{align*}
  g_\eta(t,x) 
	& = \eta^d p(t,\eta x) \leq \eta^d c_1 \left(h(t)^{-d} \wedge tf(\eta |x|)\right)
	\leq \eta^d c_1 \left(h(t)^{-d} \wedge tf(\rho |x|)\right) \\
	& \leq \eta^d c_2 p(t,\rho x)
	= \tfrac{\eta^d c_2}{\rho^d} g_\rho(t,x), \quad x\in\Rd, t\in (0,t_0].
\end{align*}
For $0<t+s\leq t_0$ and $|x|,|y|\leq h(t_0) a^{-1},$ using the 3G inequality for $p(t,x)$ given
in Lemma \ref{lem:3G} we obtain
\begin{eqnarray*}
  g_b(t,x)g_a(s,y)
	&   =  & [g_b(t,x) \vee g_a(s,y)]\cdot [g_b(t,x) \wedge g_a(s,y)] \\
  & \leq & \tfrac{c_2 b^d}{(b-a)^{d}}    [g_{b-a}(t,x) \vee g_a(s,y)]\cdot  
	         \tfrac{c_2 b^d}{a^{d}}   [g_a(t,x) \wedge g_a(s,y)] \\
  & \leq & \tfrac{c_3 c_2^2 b^{2d}}{ a^d (b-a)^d }  [g_{b-a}(t,x) \vee g_a(s,y)] g_a(t+s,x+y).
\end{eqnarray*}
Now, let $a|x|\leq h(t_0) \leq a|y|$. Then from \eqref{est_p_large_x} and Lemma \ref{lem:flocal} (with $R= h(t_0)$) we get
\begin{align*}
  a^{-d}g_a(s,y) 
	& \leq c_4 s f(a|y|)\leq c_5  s f(a|y|+a|x|) \\
	& \leq c_6 h(t+s)^{-d} \wedge [(t+s) f(a|x+y|)],
\end{align*}
 since $h(t+s)\leq h(t_0)$. This implies 
 \begin{eqnarray*}
   g_b(t,x)g_a(s,y)
   & \leq & c_7 g_{b-a}(t,x)g_a(t+s,x+y).
\end{eqnarray*}
Similarly, for  $a|y|\leq h(t_0) \leq a|x|$ we have
$$
 a^{-d}g_a(t,x) \leq c_8 t f(a|x|)\leq 
c_9 h(t+s)^{-d} \wedge [(t+s) f(a|x+y|)]
$$ and this yields
\begin{eqnarray*}
  g_b(t,x)g_a(s,y)
	& \leq & \tfrac{c_2 b^d }{a^d} g_a(t,x)g_a(s,y)
	         \leq c_{10}  g_a(s,y) g_a(t+s,x+y).
\end{eqnarray*}
It remains to consider $a^{-1}h(t_0) \leq |x|,|y|$. In this case $a^{-d}g_a(s,y)\leq c_{11} s f(a|y|)$ and  $b^{-d}g_b(t,x)\leq c_{11} t f(b|x|)$.
At first we consider $|y|\leq \frac{b-a}{a}|x|$. This implies 
$$
  f(b|x|) \leq f(a(|x|+|y|)) \leq f(a|x+y|),
$$
since $f$ is nonincreasing. If $a|x+y|\geq h(t+s)$ then $t f(a|x+y|)\leq c_{12} p(t+s,a|x+y|)$
and if $a|x+y| \leq h(t+s)$ then 
$ b^{-d} g_b(t,x) \leq c_{13} t_0f(\tfrac{b}{a} h(t_0)) \leq c_{13} t_0f(\tfrac{b}{a} h(t_0)) h(t_0)^d h(t+s)^{-d} \leq c_{14} p(t+s,a|x+y|).$
Hence
\begin{eqnarray*}
  g_b(t,x)g_a(s,y)
	& \leq & c_{15} g_a(t+s,x+y)g_a(s,y).
\end{eqnarray*}

It follows from (C) that for $|y|\geq \frac{b-a}{a}|x|$ we have
$$
  f(b|x|)f(a|y|) 
	 \leq    c_{16} f((b-a)|x|)f(a(|x|+|y|)) \leq  c_{16} f((b-a)|x|)f(a|x+y|).
$$
We note that $(b-a)|x| \geq (b-a)h(t_0)/a,$ so $f((b-a)|x|)\leq c_{17} \leq c_{18} h(t)^{-d},$ hence
\begin{eqnarray*}
  g_b(t,x)g_a(s,y)
	& \leq & c_{19} g_a(t+s,x+y)g_{b-a}(t,x).
\end{eqnarray*}
\end{proof}

\begin{lem}\label{Hoelderineq} Assume (B). Then for every $a\in (1, 2]$ there exists a constant $C=C(a)$ 
such that 
  $$
	  \int_0^t \left(\frac{r}{h(t-s)}\wedge 1 \right) p(s,y)\, ds \leq
		C (\Psi(\tfrac{1}{r}))^{\frac{1}{a}-1} \left[ \int_0^t \left( p(s,y)\right)^{a} \, ds\right]^{\frac{1}{a}}, 
  $$
	and
	$$
	  \int_0^t \left(\frac{r}{h(t-s)}\wedge 1 \right) p(t-s,y)\, ds \leq
		C (\Psi(\tfrac{1}{r}))^{\frac{1}{a}-1} \left[ \int_0^t \left( p(s,y)\right)^{a} \, ds\right]^{\frac{1}{a}},
  $$
	for all $t\leq t_0, y\in\Rd$ and $r>0$.
\end{lem}
\begin{proof} 
	Let $b$ be such that $\tfrac{1}{a} + \tfrac{1}{b} = 1$ and $r<h(t).$ We have
	\begin{align*}
	   \int_0^t \left(\frac{r}{h(t-s)}\wedge 1 \right)^b \, ds
		 &  =  \int_0^{t-\frac{1}{\Psi(\frac{1}{r})}}r^b h(t-s)^{-b}\, ds 
		        + \int_{t-\frac{1}{\Psi(\frac{1}{r})}}^t \, ds \\
		 & =  r^b \int^t_{\frac{1}{\Psi(\frac{1}{r})}}h(s)^{-b}\, ds + \frac{1}{\Psi(\frac{1}{r})} 
		 \leq  r^b \int^{ t_0}_{\frac{1}{\Psi(\frac{1}{r})}} h(s)^{-b}\, ds + \frac{1}{\Psi(\frac{1}{r})},
	\end{align*}
	and from Lemma \ref{integratingh} we get 
	\begin{align*}
	   \int_0^t \left(\frac{r}{h(t-s)}\wedge 1 \right)^b \, ds
		 & \leq  c_1 r^b \frac{1}{\Psi(\frac{1}{r})} h\left( \frac{1}{\Psi(\frac{1}{r})}\right)^{-b} + \frac{1}{\Psi(\frac{1}{r})}, \\
		 & \leq  c_2 \frac{1}{\Psi(\frac{1}{r}) }
	\end{align*}
	provided $b>\alpha_2$, which follows from the fact that $\alpha_2<2$ and $a\leq 2$. If $r>h(t)$ then we just use the fact that in this case 
	$t\leq \frac{1}{\Psi(1/r)}.$
	From H\"older inequality we get
  \begin{eqnarray*}
	  \int_0^t \left(\frac{r}{h(t-s)}\wedge 1 \right) p(s,y) \, ds
		& \leq & \left[ \int_0^t \left( p(s,y)\right)^{a} \, ds\right]^{\frac{1}{a}}
		\left[ \int_0^t \left(\frac{r}{h(t-s)}\wedge 1 \right)^b \, ds \right]^{\frac{1}{b}}, 
	\end{eqnarray*}
	and
  \begin{eqnarray*}
	  \int_0^t \left(\frac{r}{h(t-s)}\wedge 1 \right) p(t-s,y) \, ds
		& \leq & \left[ \int_0^t \left( p(t-s,y)\right)^{a} \, ds\right]^{\frac{1}{a}}
		\left[ \int_0^t \left(\frac{r}{h(t-s)}\wedge 1 \right)^b \, ds \right]^{\frac{1}{b}},
	\end{eqnarray*}
	and the lemma follows.

\end{proof}

We recall that
\begin{align} \label{eq:aux_def_Gn}
  G_n(t,x) = \min\left\{H(t)^{-d},tf\left(\tfrac{|x|}{4}\right)\right\} + H(t)^{-d}\left(1+\frac{|x|}{H(t)}\right)^{-n}.
\end{align}

 It is useful to note that for every fixed $n$ and $t$ the function $\R^d \ni x \mapsto G_n(t,x)$ is bounded and radial non-increasing. 

\begin{lem}\label{Gcomp_p} If (A) and (B) hold, then for every $n\geq d+\alpha_2$ we have
  $$
	  G_n(t,x) \approx \min\{h(t)^{-d},tf(|x|)\} \approx p(t,x), \quad |x|<2, t\in (0,t_0], 
	$$
	and
	\begin{equation*}
	  G_n(t,y-x) \geq C \min\left\{h(t-s)^{-d}\left(1+\tfrac{|w-x|}{h(t-s)}\right)^{-n},p(s,y-w)\right\}, 
	\end{equation*}
	for all $x,y,w\in\Rd$, and $0<s<t\leq t_0$.
	
\end{lem}
\begin{proof}
  For $|x|\leq  h(t)$ we have (see \eqref{f(h)})
	$$
	  tf(|x|) \geq t f( h(t)) \geq c_1 h(t)^{-d}
	$$
	so $\min\{h(t)^{-d},tf(|x|)\}\geq c_2 h(t)^{-d}$, and, since
	$(1+\frac{|x|}{h(t)})^{-n} \leq 1$ we get $G_n(t,x)\leq c_3 \min\{h(t)^{-d},tf(|x|)\}$
	in this case. For $ h(t)\leq |x| \leq 2$ we still have $G_n(t,x)\leq 2h(t)^{-d}$ and using
	\eqref{f(h)} and \eqref{eq:doubling_kappa}
	we get
	\begin{align*}
	  tf(|x|) 
		& \geq  c_4 h(t)^{-d} \frac{f(|x|)}{f( h(t))} 
		\geq c_5 h(t)^{-d} 
		         \left(\frac{h(t)}{|x|}\right)^{d+\alpha_2} \\
		& \geq  c_5 h(t)^{-d} \left(\frac{ h(t)}{|x|}\right)^{n},
	\end{align*}
	provided $n\geq d+\alpha_2$. This yields
	$$
	  G_n(t,x) \leq tf\left(\tfrac{|x|}{4}\right) + h(t)^{-d}\left(\frac{h(t)}{|x|}\right)^{n} \leq c_6 t f(|x|),
	$$
	since from \eqref{eq:doubling_kappa} we have $f(\frac{|x|}{4})\leq c f(|x|)$. This 
	gives $G_n(t,x)\leq c_7 \min\{h(t)^{-d},tf(|x|)\}$. The opposite
	inequality follows directly from the definition of $G_n$.
	For every $t>s>0$ and $x,y,w\in\Rd,$ using Lemma \ref{lem:doubling_h} we obtain
\begin{align*}
  G_n(t,y-x) 
	& \geq  \min\{h(t)^{-d},tf\left(\tfrac{|y-x|}{4}\right)\} +
	         2^{-n} \min\{h(t)^{-d},h(t)^{n-d}|y-x|^{-n}\} \\
	& \geq  2^{-n} \min\{ h(t)^{-d},h(t)^{n-d}|y-x|^{-n} + tf\left(\tfrac{|y-x|}{4}\right)\} \\
	& \geq  c_8 \min\{ h(\tfrac{t}{2})^{-d},h(t)^{n-d}\left(|w-x|\vee |y-w|\right)^{-n} 
	        + tf\left(|w-x|\vee |y-w|\right) \} \\
	& \geq  c_8 \min\{ h((t-s)\vee s)^{-d},h(t)^{n-d}\left(|w-x|\right)^{-n} 
	         + tf\left(|w-x|\right),\\
	&       h(t)^{n-d}\left(|y-w|\right)^{-n} 
	         + tf\left(|y-w|\right) \} \\
	& \geq  c_8 \min\{ h(t-s)^{-d},h(s)^{-d},h(t-s)^{n-d}|w-x|^{-n},sf(|y-w|)\} \\
	& \geq  c_9 \min\{h(t-s)^{-d}\left(1+\tfrac{|w-x|}{h(t-s)}\right)^{-n},p(s,y-w)\}.
\end{align*}
\end{proof}

\begin{lem}\label{lem_DiffEst} If (A) and (B) hold then for every $n\geq d+\alpha_2$ 
there exists $C=C(n)$ such that
  \begin{equation}\label{DiffEst}
    |p(t,x) - p(t,y)| \leq C \left(\frac{|y-x|}{h(t)}\wedge 1\right) 
	  \left( G_n(t,x) + G_n(t,y) \right),\quad t\in (0,t_0],\, x,y\in\Rd.
  \end{equation}
\end{lem}
\begin{proof}
It follows from Lemma \ref{pderest} that 
\begin{equation}\label{pDerEst}
  \left|\frac{\partial}{\partial x_i} p(t,x)\right| \leq c_1 (h(t))^{-1} G_n(t,2x),
\end{equation}
for all $x\in\Rd$, $t\in (0,t_0)$ and every $i\in\{1,...,d\}.$ 

If $|y-x|>h(t)$ then we just estimate the difference of functions by their sum and \eqref{DiffEst} follows from \eqref{eq:ptx_est}.
If $|y-x|\leq  h(t)$ then using the Taylor expansion and \eqref{pDerEst} we get
\begin{eqnarray*}
  |p(t,x) - p(t,y)| 
	& \leq & |x-y| \cdot \sup_{\zeta\in [0,1]} |\nabla_x p(t,x+\zeta(y-x))|\\
	& \leq & c_2 |x-y|\frac{1}{h(t)} \sup_{\zeta\in [0,1]} G_n(t,2(x+\zeta(y-x))) \\
	& \leq & c_3 \frac{|y-x|}{h(t)} \left(G_n(t,x)+G_n(t,y)\right),
\end{eqnarray*}
since $G_n(t,\cdot)$ is isotropic non-increasing function. 
\end{proof}

\begin{proof}[Proof of Theorem \ref{THE2}.]
First let us observe that by monotonicity of $n\mapsto G_n$ we may and do assume that $n>d+\alpha_2$.
If $|z-x|\geq h(t)$ then the inequality \eqref{Hoelkernel} follows from the estimate \eqref{tildep_est} and
Lemma \ref{p_t_est} so we assume further that $|z-x| < h(t).$  
Let
$$
  \phi (t,x,y) = \int_0^t \int_{\Rd} p(t-s,w-x)q(w)\widetilde{p}(s,w,y)\, dwds.
$$
It follows from \eqref{pertp} that
$$
\widetilde{p}(t,x,y) = p(t,y-x) + \phi(t,x,y).
$$
From Lemma \ref{lem_DiffEst} we get
\begin{align*}
  |\phi(t,z,y) - & \phi(t,x,y)|  \leq c_1 \int_0^t \int_{\Rd} 
	 \left(\frac{|z-x|}{h(t-s)}\wedge 1\right) \left(G_n(t-s,w-x)+G_n(t-s,w-z)\right)\\
	 & \times |q(w)|\widetilde{p}(s,w,y)\, dw ds.
\end{align*}

Since $\widetilde{p}(s,w,y) \leq c_2 p(s,y-w),$ $s < t,$ it is 
sufficient to estimate
$$
  \int_0^t\int_{\Rd} \left(\frac{|z-x|}{h(t-s)}\wedge 1\right) G_n(t-s,w-x) |q(w)| p(s,y-w)\, dwds
$$

Using Lemma \ref{Gcomp_p} we obtain
\begin{align*}
   & \int_0^t \int_{\Rd} \left(\frac{|z-x|}{h(t-s)}\wedge 1\right) 
	     G_n(t-s,w-x) |q(w)| p(s,y-w)\, dwds \\
	 & \leq   c_3 \int_0^t \int_{\Rd} \left(\frac{|z-x|}{h(t-s)}\wedge 1\right) 
	         p\left(t-s,\tfrac{w-x}{4}\right) |q(w)| p(s,y-w)\, dwds \\
	   + &   c_4 \int_0^t \int_{|w-x|\geq 2} \left(\frac{|z-x|}{h(t-s)}\wedge 1\right) 
	         h(t-s)^{-d}\left(1 + \frac{|w-x|}{h(t-s)}\right)^{-n}   |q(w)| p(s,y-w)\, dwds \\
	&       := c_3 A + c_4 B.
\end{align*}
From Proposition \ref{prop:4G} we get
\begin{eqnarray*}
  A
	& \leq & c p\left(t,\tfrac{y-x}{4}\right) \int_0^t \int_{\Rd} \left(\frac{|z-x|}{h(t-s)}\wedge 1\right) 
	          |q(w)| \left(p\left(s,\tfrac{3}{4}(y-w)\right) + p\left(t-s,\tfrac{1}{4}(w-x)\right)\right)\, dwds \\
\end{eqnarray*}
and from Lemma \ref{Hoelderineq} we have
\begin{align*}
  A 
	& \leq  c_5 p\left(t,\tfrac{y-x}{4}\right) \left(\Psi(\tfrac{1}{|z-x|})\right)^{\frac{1}{a}-1} \int_{\Rd} |q(w)| \\
	&       \times \left[ \left(\int_0^t \left(p(s,\tfrac{w-x}{4})\right)^a\, ds\right)^{\frac{1}{a}} + \left(\int_0^t \left(p(s,\tfrac{3}{4}(y-w))\right)^a\, ds\right)^{\frac{1}{a}} \right]\, dw \\
	& \leq  c_6 p\left(t,\tfrac{y-x}{4}\right) \left(\Psi(\tfrac{1}{|z-x|})\right)^{\frac{1}{a}-1} 
	         \sup_{x\in\Rd} \int_{\Rd} |q(w)|  \left[\int_0^t \left(p(s,\tfrac{w-x}{4})\right)^a\, ds\right]^{\frac{1}{a}}\, dw, \\
	& = 4^d c_6 p\left(t,\tfrac{y-x}{4}\right) \left(\Psi(\tfrac{1}{|z-x|})\right)^{\frac{1}{a}-1} 
	         \sup_{x\in\Rd} \int_{\Rd} |q(4w)|  \left[\int_0^t \left(p(s,w-\tfrac{x}{4})\right)^a\, ds\right]^{\frac{1}{a}}\, dw,
\end{align*}
so by Corollary \ref{CorKato_Eq} we get $A\leq c_6 p\left(t,\tfrac{y-x}{4}\right) \left(\Psi(\tfrac{1}{|z-x|})\right)^{\frac{1}{a}-1}.$

Lemma \ref{Gcomp_p} yields
\begin{eqnarray*}
  B
	& \leq & c_7 \, G_n(t,y-x) \int_0^t \int_{|w-x|\geq 1} \left(\tfrac{|z-x|}{h(t-s)}\wedge 1\right)
	          |q(w)| \\
 &       & \times \left(p(s,y-w) + h(t-s)^{-d}\left(1+\tfrac{|w-x|}{h(t-s)}\right)^{-n}\right)\, dwds \\
 &   =   & c_7 \, G_n(t,y-x) \int_{|w-x|\geq 1} |q(w)| \left( \int_0^t  
           \left(\tfrac{|z-x|}{h(t-s)}\wedge 1\right) p(s,y-w)\, ds \right) \, dw  \\
&        & + \,\, c_7\, G_n(t,y-x)  \int_{|w-x|\geq 1} |q(w)|  \left( \int_0^t  
           \left(\tfrac{|z-x|}{h(t-s)}\wedge 1\right) h(t-s)^{-d}\left(1+\tfrac{|w-x|}{h(t-s)}\right)^{-n}\, ds \right)\, dw.
\end{eqnarray*}
The first integral we estimate exactly as above and for the second we use the decomposition of $\Rd$ into unit cubes $K_\beta,$
$\beta\in\N$ and define $I=\{\beta\in\N:\: K_\beta\cap B(x,1) = \emptyset\}$. Using Corollary \ref{Cor_qint} we get
\begin{eqnarray*}
  &      & \int_{|w-x|\geq 1} |q(w)| \left( \int_0^t \left(\tfrac{|z-x|}{h(t-s)}\wedge 1\right) h(t-s)^{-d} \left(1+\tfrac{|w-x|}{h(t-s)}\right)^{-n}\, ds \right)\, dw \\
	& \leq & \int_{|w-x|\geq 1} |q(w)| \int_0^t |z-x| h(t-s)^{n-d-1} |w-x|^{-n} \, dsdw \\
	&   =  & |z-x| \int_0^t h(s)^{n-d-1}\, ds \int_{|w-x|\geq 1} |q(w)| |w-x|^{-n}\, dw \\
	& \leq & c_8 t h(t)^{n-d-1} |z-x| \left(\int_{|w-x|\leq 1+\sqrt{d}} |q(w)|\, dw + \sum_{\beta\in I}\int_{K_\beta}|q(w)| |w-x|^{-n}\, dw\right) \\
	& \leq & c_9 t h(t)^{n-d-1} |z-x| \sup_{y\in\Rd} \int_{|w-y|\leq 1+\sqrt{d}} |q(w)|\, dw	. 
\end{eqnarray*}
Finally, using Lemma \ref{lem:doubling_h} and (B) we obtain $$\Psi\left(\tfrac{1}{|z-x|}\right) \geq c_{10}|z-x|^d f(|z-x|) \geq c_{11} |z-x|^{-\alpha_1}, \quad \text{for}  \quad |z-x|\leq 2,$$ hence
\begin{eqnarray*}
  A + B
	& \leq & c_{12}\, G_n(t,y-x) |z-x|^{\tfrac{\alpha_1}{b}}.
\end{eqnarray*}
This and Lemma \ref{lem_DiffEst} yield
$$
  |\widetilde{p}(t,z,y) - \widetilde{p}(t,x,y)| \leq c_{13} \left(\tfrac{|z-x|}{h(t)}\wedge 1\right)^{\tfrac{\alpha_1}{b}}\left(G_n(t,y-z) + G_n(t,y-x)\right).
$$

Now we assume that $\alpha_1>1$ and $a>\tfrac{\alpha_1}{\alpha_1-1}$ and prove the differentiability 
of $\widetilde{p}(\cdot,y).$

First we prove that
there exists $C$ such that for all $x,y\in\Rd$ we have 
\begin{equation}\label{DiffCond}
  \int_0^t \int_{\Rd} \frac{G_n(t-s,z-x)}{h(t-s)} |q(z)| p(s,y-z)\, dzds \leq C t^{\frac{1}{b}}h(t)^{-1}G_n(t,y-x).
\end{equation}
Similarly as above, we define
$$
  A:= \int_0^t \int_{\Rd} \frac{1}{h(t-s)} 
	         p(t-s,\tfrac{z-x}{4}) |q(z)| p(s,y-z)\, dzds,
$$
and
$$
  B:= \int_0^t \int_{|z-x|\geq 2} 
	         h(t-s)^{-d-1}\left(1 + \frac{|z-x|}{h(t-s)}\right)^{-n}   |q(z)| p(s,y-z)\, dzds
$$

It follows from Proposition \ref{prop:4G} that
$$
  A \leq c_{14} p\left(t,\tfrac{y-x}{4}\right) \int_0^t \int_{\Rd} \frac{|q(z)|}{h(t-s)} \left(p\left(s,\tfrac{3(y-z)}{4}\right)+p\left(t-s,\tfrac{z-x}{4}\right)\right)\, dzds 
$$
and by H\"older inequality, \eqref{extraKato} and Lemma \ref{integratingh} we get
\begin{eqnarray*}
  A 
	& \leq & c_{15} p\left(t,\tfrac{y-x}{4}\right) \int_{\Rd} |q(z)| \left(\int_0^t (h(s))^{-b}\, ds\right)^{\frac{1}{b}} \\
	&      & \cdot
	\left(\left(\int_0^t \left(p\left(s,\tfrac{3(y-z)}{4}\right)\right)^a\, ds \right)^{\frac{1}{a}} + \left(\int_0^t\left(p\left(t-s,\tfrac{z-x}{4}\right)\right)^a\, ds\right)^{\frac{1}{a}}\right)\, dz \\
	& \leq & c_{16} \frac{t^{\frac{1}{b}}}{h(t)} p\left(t,\tfrac{y-x}{4}\right), 
\end{eqnarray*}
provided $b<\alpha_1$.
In the same way as above we obtain
\begin{eqnarray*}
  B 
	& \leq & c_{17} G_n(t,y-x) \int_0^t \int_{|z-x|\geq 1} \frac{|q(z)|}{h(t-s)} \left(p(s,y-z) + h(t-s)^{-d}\left(1+\tfrac{|z-x|}{h(t-s)}\right)^{-n}\right)\, dzds \\
	&  =  &  c_{17} G_n(t,y-x) \int_{|z-x|\geq 1} |q(z)| \left( \int_0^t  \frac{p(s,y-z)}{h(t-s)} \, ds 
	 +   \int_0^t \frac{|z-x|^{-n}}{(h(t-s))^{d+1-n}}\, ds \right)\, dz,
\end{eqnarray*}
For the first integral over time we use the H\"older inequality to get
$$
  \int_0^t  \frac{p(s,y-z)}{h(t-s)} \, ds \leq \left( \int_0^t (h(t-s))^{-b}\, ds \right)^{\frac{1}{b}}
	\left(\int_0^t \left(p(s,y-z)\right)^a\, ds\right)^{\frac{1}{a}}.
$$
Then \eqref{extraKato} and Lemma \ref{integratingh} yield
$$
  \int_{|z-x|\geq 1} |q(z)| \left(  \int_0^t  \frac{p(s,y-z)}{h(t-s)} \, ds\right) dz \leq \frac{c_{18}t^{\frac{1}{b}}}{h(t)}.
$$
For the second integral we use Lemma \ref{integratingh} and get
$$
  \int_0^t (h(t-s))^{-d-1+n}\, ds \leq c_{19} t(h(t))^{-d-1+n},
$$
so it is sufficient to estimate
$$
  \int_{|z-x|\geq 1} |q(z)| |z-x|^{-n} dz,
$$
by a constant, which we do exactly as above. Finally, we obtain
$$ 
  A + B \leq c_{20} G_n(t,y-x) \cdot t^{\frac{1}{b}}h(t)^{-1}, 
$$
and \eqref{DiffCond} for $t\in (0,t_0]$ follows.
Now, for every $r>0$ we have
\begin{align*}
 & \frac{\phi (t,x+re_i,y)  - \phi(t,x,y) }{r}  \\
& = \int_0^t \int_{\Rd} \frac{p(t-s,w-x-re_i)-p(t-s,w-x)}{r}q(w)\widetilde{p}(s,w,y)\, dwds \\
& = \int_0^t \int_{\Rd} \frac{1}{r}\left(\int_0^r \frac{\partial }{\partial x_i}p(t-s,w-x-\rho e_i)\, d\rho \right) 
q(w)\widetilde{p}(s,w,y)\, dwds.
\end{align*}
If $r\leq \tfrac{1}{2} h(t-s)$ then it follows from Corollary \ref{pleqp} that $$G_n(t-s,w-x-\rho e_i) \leq c_{21} G_n(t-s,w-x),$$ hence 
$$
  \indyk{\left(0,t-\frac{1}{\Psi(1/r)}\right)}(s) \frac{1}{r} \left|\int_0^r \frac{\partial }{\partial x_i}p(t-s,w-x-\rho e_i)\, d\rho \right|
	\leq c_{22} h(t-s)^{-1} G_n(t-s,w-x).
$$
This and \eqref{DiffCond} yields (by dominated convergence) that
\begin{align*}
  \lim_{r\to 0}\int_0^{t-\frac{1}{\Psi(1/r)}} &
	\int_{\Rd} \frac{p(t-s,w-x-re_i)-p(t-s,w-x)}{r}q(w)\widetilde{p}(s,w,y)\, dwds
	\\
	& = -\int_0^{t} 
	\int_{\Rd} \frac{\partial}{\partial x_i} p(t-s,w-x) q(w)\widetilde{p}(s,w,y)\, dwds.
\end{align*}
Furthermore, using H\"older inequality we get
\begin{align*}
  \int_{t-\frac{1}{\Psi(1/r)}}^t &
	\int_{\Rd} \left|\frac{p(t-s,w-x-re_i)-p(t-s,w-x)}{r}\right| |q(w)|\widetilde{p}(s,w,y)\, dwds \\
	& = \int_0^{\frac{1}{\Psi(1/r)}} 
	\int_{\Rd} \left|\frac{p(s,w-x-re_i)-p(s,w-x)}{r}\right| |q(w)|\widetilde{p}(t-s,w,y)\, dwds \\
	& \leq c_{23} h(t/2)^{-d} \int_0^{\frac{1}{\Psi(1/r)}} 
	\int_{\Rd} \frac{p(s,w-x-re_i)+p(s,w-x)}{r} |q(w)|\, dwds \\
	& \leq c_{24} h(t/2)^{-d} \sup_{x\in\Rd} \int_0^{\frac{1}{\Psi(1/r)}} 
	\int_{\Rd} \frac{p(s,w-x)}{r} |q(w)|\, dwds \\
	& \leq c_{24} h(t/2)^{-d} \sup_{x\in\Rd} \int_{\Rd} |q(w)| \left(\int_0^{\frac{1}{\Psi(1/r)}} 
	 \left[p(s,w-x)\right]^a \, ds \right)^{\frac{1}{a}} \, dw \\
	& \times \frac{1}{r} \left(\int_0^{\frac{1}{\Psi(1/r)}}  
	  \, ds \right)^{\frac{1}{b}} \to 0
\end{align*}
as $r\to 0,$ since $\frac{1}{r\Psi(1/r)^{1/b}} \leq c_{25} r^{\frac{\alpha_1}{b}-1},$ and $\alpha_1/b-1 = \alpha_1-\frac{\alpha_1}{a}-1>0.$ This yields the existence of 
$$
  \frac{\partial}{\partial x_i} \tilde p(t,x,y) = -\frac{\partial}{\partial x_i} p(t,y-x)-\int_0^{t} 
	\int_{\Rd} \frac{\partial}{\partial x_i} p(t-s,w-x) q(w)\widetilde{p}(s,w,y)\, dwds,
$$
and the estimate \eqref{DerEst} follows from \eqref{DiffCond} and \eqref{pDerEst}.
\end{proof}

 For the reader convenience, we also give a short justification of Corollary \ref{cor:reg_sem}.

\begin{proof}[Proof of Corollary \ref{cor:reg_sem}]
We first show (1). Fix $\beta \in (0,\alpha_1/2)$, $t_0>0$, $a \geq \alpha_1/(\alpha_1-\beta)$ and consider $q\in\KatoC_a$.
Then by \eqref{Hoelkernel} for every $n>0$ and for every $\varphi\in L^p(\R^d)$, $x,y \in \R^d$ and $t \in (0,t_0)$, 
\begin{equation*}
  |\widetilde{P}_t\varphi(x) - \widetilde{P}_t \varphi(y)| \leq c 
	   \left(\tfrac{|x-y |}{h(t)}\wedge 1\right)^{\tfrac{\alpha_1}{b}}\sup_{w \in \R^d} \int_{\R^d} G_n(t,w-z) \varphi(z) dz,
\end{equation*}
with $c=c(t_0)$. Since 
$$
\frac{\alpha_1}{b} = \alpha_1\left(1-\frac{1}{a}\right) \geq \alpha_1\left(1-\frac{\alpha_1-\beta}{\alpha_1}\right) = \beta,
$$
we are left to find an upper bound for the integral on the right hand side. To this end, we choose $q=p/(p-1)$, $n = (d+1)/q$ for $p \in [1,\infty)$, $q =\infty$ for $p=1$, and $q=1$ for $p=\infty$ (we use the standard convention that $a/\infty = 0$ whenever $a>0$). Observe that by \eqref{eq:aux_def_Gn} and \eqref{eq:ptx_est} there exists $c_1 = c_1(t_0,p)$ such that
\begin{align*}
\int_{\R^d} G_n(t,w-z) \varphi(z) dz  \leq & \, c_1 \int_{\R^d} p(t,(w-z)/4) \varphi(z) dz \\ 
& + h(t)^{-d} \int_{\R^d} \left(1+\frac{|w-z|}{h(t)}\right)^{-\frac{d+1}{q} }\varphi(z)dz.
\end{align*}

Suppose first that $p \in [1,\infty)$. By using the H\"older inequality and the fact that $p(t,x)$ is the density of a probability measure satisfying estimates \eqref{eq:ptx_est}, we get
\begin{align*}
\int_{\R^d} & p(t,(w-z)/4) \varphi(z) dz \\ & \leq \left(\int_{\R^d} p(t,(w-z)/4) dz\right)^{1/q} \left(\int_{\R^d} p(t,(w-z)/4) |\varphi(z)|^p dz\right)^{1/p} \\
            & \leq c_2 h(t)^{-d/p} \left\|\varphi\right\|_p,
\end{align*}
with $c_2=c_2(t_0,p)$. Similarly,
$$
\int_{\R^d} \left(1+\frac{|w-z|}{h(t)}\right)^{-\frac{d+1}{q} }\varphi(z)dz \leq \left(\int_{\R^d} \left(1+\frac{|w-z|}{h(t_0)}\right)^{-d-1}dz \right)^{1/q} 
\left\|\varphi\right\|_p.
$$
For $p=1$ and $p=\infty$ the corresponding bounds follow directly.
Collecting all the above estimates, we conclude that 
$$
  |\widetilde{P}_t\varphi(x) - \widetilde{P}_t \varphi(y)| \leq c_3 h(t)^{-d/p-\beta}  |x-y|^\beta \left\|\varphi\right\|_p, \quad x, y \in \R^d, \ t \in (0,t_0),
$$
where $c_3 = c_3(t_0,p)$. This completes the proof of part (1). 

Assertion (2) is a straightforward consequence of \eqref{DerEst} and the estimate of the integral $\int_{\R^d} G_n(t,x-z) \varphi(z) dz$ that we provided in the proof of part (1). 
\end{proof}


\bibliographystyle{abbrv}
\bibliography{pert_bib}

\end{document}